\documentclass[dvipsnames,A4paper,11pt]{article}

\usepackage[margin=1in]{geometry}
\usepackage[utf8]{inputenc}
\usepackage{graphicx,fullpage,paralist}
\usepackage{amsmath,amssymb,tabu,amsthm}
\usepackage{dsfont}
\usepackage{bm}
\usepackage{array}
\newcolumntype{C}{>{$}c<{$}}
\usepackage{comment,hyperref}
\usepackage{booktabs}
\usepackage{subcaption}
\usepackage[font={footnotesize}]{caption}
\usepackage{mathtools}
\usepackage{thmtools}
\usepackage{tikz}
\usepackage{tikz-network}
\usepackage{pgfplots}
\usepackage{varwidth}
\usepackage{xstring}
\pgfplotsset{compat=1.10}
\usepgfplotslibrary{fillbetween}
\usetikzlibrary{backgrounds}
\usetikzlibrary{patterns}
\usetikzlibrary{arrows}
\usetikzlibrary{calc}
\usetikzlibrary{matrix,decorations.pathreplacing,positioning,fit}
\usepackage{enumitem}
\usepackage{bbm}
\usepackage{blkarray}
\usepackage{csquotes}
\usepackage{stmaryrd}
\usepackage[numbers,sort]{natbib}

\newcommand{\R}{\mathbb{R}}
\newcommand{\N}{\mathbb{N}}

\newcommand{\RR}{\mathbb{R}}
\newcommand{\NN}{\mathbb{N}}

\newcommand{\PP}{\mathbf{P}}
\newcommand{\EE}{\mathbf{E}}
\newcommand{\E}{\mathbf{E}}

\newcommand{\ind}{\mathbbm{1}}
\renewcommand{\partial}{\normalfont\text{deg}}

\newcommand{\td}{f} 

\newtheorem{theorem}{Theorem}[section]
\newtheorem{lemma}[theorem]{Lemma}

\newtheorem{proposition}[theorem]{Proposition}
\newtheorem{definition}[theorem]{Definition}

\newtheorem{assumption}[theorem]{Assumption}

\theoremstyle{definition}

\numberwithin{equation}{section}

\usepackage[textsize=small,textwidth=2cm]{todonotes}

\newcommand{\david}[1]{\todo[color=green!25!white,inline]{David: #1}}

\usepackage{multirow}
\usepackage[noend]{algpseudocode}
\usepackage{algorithm,algorithmicx}

\newlength{\algofontsize}
\hypersetup{
	colorlinks,
	linkcolor={red!50!black},
	citecolor={blue!50!black},
	urlcolor={blue!80!black}
}

\allowdisplaybreaks

\begin{document}
\algrenewcommand\algorithmicrequire{\textbf{Input:}}
\algrenewcommand\algorithmicensure{\textbf{Output:}}
	
\title{Mean-field Concentration of Opinion Dynamics in Random Graphs 
\vspace{.5cm}
}
	
\author{Javiera Guti\'errez-Ramírez
\thanks{Department of Mathematical Engineering, Universidad de Chile.}
\and David Salas	
\thanks{Institute of Engineering Sciences, Universidad de O'Higgins.}
\and V\'ictor Verdugo
\thanks{Institute for Mathematical and Computational Engineering, Pontificia Universidad Católica de Chile.}
\thanks{Department of Industrial and Systems Engineering, Pontificia Universidad Católica de Chile.}
}

\date{}

\maketitle
\thispagestyle{empty}
\begin{abstract}
Opinion and belief dynamics are a central topic in the study of social interactions through dynamical systems. In this work, we study a model where, at each discrete time, all the agents update their opinion as an average of their intrinsic opinion and the opinion of their neighbors. While it is well-known how to compute the stable opinion state for a given network, studying the dynamics becomes challenging when the network is uncertain. Motivated by the task of finding optimal policies by a decision-maker that aims to incorporate the opinion of the agents, we address the question of how well the stable opinions can be approximated when the underlying network is random. 

We consider Erd\H{o}s-Rényi random graphs to model the uncertain network.  Under the connectivity regime and an assumption of minimal stubbornness, we show the expected value of the stable opinion $\E(x(G,\infty))$ concentrates, as the size of the network grows, around the stable opinion $\bar{x}(\infty)$ obtained by considering a mean-field dynamical system, i.e., averaging over the possible network realizations. For both the directed and undirected graph model, the concentration holds under the $\ell_{\infty}$-norm to measure the gap between $\E(x(G,\infty))$ and $\bar{x}(\infty)$. We deduce this result by studying a mean-field approximation of general analytic matrix functions. The approximation result for the directed graph model also holds for any $\ell_{\rho}$-norm with $\rho\in (1,\infty)$, under a slightly enhanced expected average degree. 

\end{abstract}

\section{Introduction}

The rise of large social networks, online platforms, and massive decentralized complex systems has presented unprecedented modeling and algorithmic challenges stemming from the vast volume of data, the distributed nature of information, and the dynamic interaction of self-interested agents. One of the most fundamental questions in studying social networks and decentralized systems is understanding the dynamics of the agent's beliefs or opinions about specific topics over time \cite{degroot1974reaching,FriedkinJohnsen1999Influence}. This question has been treated extensively in the last two decades, including the impact of susceptibility considerations, influence, external shocks, agent incentives, and network design, among many others (see, e.g., \cite{Acemouglu2013Opinion,Acemoglu2011Learning,hegselmann2002opinion,Kempe2015Spread,Bala2000Noncooperative,Ballester2005Who,Chu2023Non-Markovian,Geethu2021Controllability, Jia2015Power}.) 

A natural modeling approach is to encode the network structure using a graph $G=(V,E)$, which can be undirected or directed, depending on the network characteristics. Every node represents an agent, and the beliefs dynamics are given by a discrete dynamical system of the form 
where $x(G,t)\in [0,1]^V$ encodes the opinion of the agents at time $t$ and $F_G$ is a function that captures the network structure and agents' information. In this work, we consider a function $F_G$ that for each agent $i\in V$ averages the agent's intrinsic opinion $x_i(0)\in [0,1]$ with the opinions of the agents that interact with $i$ in the network (i.e., neighbors of $i$ in $G$). This average is weighted according to a {\it condescendence} value $\alpha_i\in [0,1]$ (i.e., $1-\alpha_i$ quantifies the {\it stubbornness} of the agent) measuring how likely is $i$ to update his opinion according to the ones of its neighbors. This model was introduced by Ghaderi and Srikant \cite{GhaderiSrikant2014Opinion}, and generalizes other classic models in the literature (e.g., DeGroot~\cite{degroot1974reaching} and Friedkin-Johnsen~\cite{FriedkinJohnsen1999Influence}). We refer to this model as the {\it opinion dynamics}; see Section \ref{sec:Models} for the formal model description. Under mild assumptions, the opinion dynamics admit a stable solution $x(G,\infty)$, i.e., a fixed point of the dynamical system \cite{GhaderiSrikant2014Opinion}.

Besides the intrinsic difficulty of understanding these dynamics in a complex network, the task becomes even more challenging when a decision-maker needs to implement a management policy under uncertainty considerations over the graph structure \cite{jiang2023opinion,bauso2016opinion,musco2018minimizing,koshal2016distributed}. To address this problem, we study opinion dynamics through the lens of random graphs and stochastic optimization. Namely, under a stable state $x(G,\infty)$, the decision-maker wants to solve an optimization problem of the form
\begin{equation}\label{eq:OptProblem}
    \min_{\theta\in \Theta} \E (f(x(G,\infty),\theta)).
\end{equation}
Under standard technical considerations, e.g., continuity of $f$ on the second variable and compactness of the parameters set $\Theta$, Problem \eqref{eq:OptProblem} admits a solution. 
As the possible number of graphs realization can be huge, sampling methods or mini-batch procedures can be applied to Problem \eqref{eq:OptProblem} (see, e.g., \cite{byrd2016stochastic,li2014efficient}). However, before trying these complex and (usually) computationally very expensive methods, a common practice in stochastic optimization is first to solve the \textit{mean value problem} associated with \eqref{eq:OptProblem}, which is given by
\begin{equation}\label{eq:EVV}
    \min_{\theta\in \Theta}  f(\E(x(G,\infty)),\theta).
\end{equation}
In many cases, one can show that the solution of \eqref{eq:EVV} is already optimal or near-optimal for \eqref{eq:OptProblem}, which is translated as a small \textit{value of stochastic solution} (see, e.g., \cite[Chapter 4]{Birge2011Introduction}); for example, if $f$ is affine in the first variable, that is $f(x,\theta) = \langle a(\theta),x\rangle + b(\theta)$, then \eqref{eq:EVV} and \eqref{eq:OptProblem} coincide.
For this strategy to be advantageous, one should be able to handle the mean value problem \eqref{eq:EVV} efficiently. 
In our setting, when sampling methods are used for this computation, the advantage of Problem \eqref{eq:EVV} is not clear versus a {\it sample average approximation} (SAA) of Problem \eqref{eq:OptProblem}; for the description of the SAA method see, e.g., \cite{Homem-de-Mello2014MonteCarlo,Shapiro2021Lectures}. Indeed, for each sampled graph $\hat{G}$, the computation of $x(\hat{G},\infty)$ is very expensive since it requires to invert a large matrix; see, e.g., \cite{GhaderiSrikant2014Opinion} or Proposition \ref{prop:LimitOpinion} below.

In contrast, if one can compute efficiently an approximation $\bar{x}$ of $\E(x(G,\infty))$, then one could solve the auxiliary problem
\begin{equation}\label{eq:EVV-approx}
    \min_{\theta\in \Theta}  f(\bar{x},\theta).
\end{equation}
With this approach, if $f$ is $L$-Lipschitz with respect to the $\ell_{\infty}$-norm on the first variable, then, for any solution $\theta^{\star}$ of \eqref{eq:EVV-approx} and any $\theta\in \Theta$ we have
\begin{align*}
    |f(\E(x(G,\infty)),\theta^{\star})-f(\bar{x},\theta)|&\leq L\|\E(x(G,\infty)) - \bar{x}\|_{\infty},
\end{align*}
that is, $\theta^{\star}$ is near optimal for the mean value problem, with an $L\|\E(x(G,\infty)) - \bar{x}\|_{\infty}$ error. As the number of agents $n$ grows, sampling the graphs for \eqref{eq:OptProblem} or even \eqref{eq:EVV} becomes computationally prohibitive. However, the approximate problem \eqref{eq:EVV-approx} can be still tractable. Then, one of the main questions driving this work is the following: {\it Can we use a simple and tractable approximation $\bar{x}$ for which the error $\|\E(x(G,\infty)) - \bar{x}\|_{\infty}$ vanishes as $n\to\infty$?} 

\subsection{Our Contribution and Results}

The opinion dynamics can be compactly written as $x(G,t+1) = H(G)x(G,t) + Bx(0)$, where $B$ encodes the agents' stubbornness and $H(G)$ is linear in the adjacent matrix $A(G)$ of the graph $G$. In this case, the stable opinion can be computed as $x(G,\infty) =  (I-H(G))^{-1}Bx(0)$. We first consider the Erd\H{o}s-R\'enyi undirected random graph model assuming that the probability $p$ of sampling an edge is over the connectivity threshold and that the agents' condescendence is uniformly bounded by some factor $\bar{\alpha}<1$. We show that it holds $\|\E(x(G,\infty)) - \bar{x}\|_{\infty}\to 0$ when $n\to \infty$, where $\bar{x}= (I - \EE(H(G)))^{-1}Bx(0)$ is the stable solution of the mean-field dynamics, i.e., the stable solution of the (deterministic) opinion dynamics obtained by using the expected adjacency matrix of the random graph $G$ (Theorem \ref{thm:Main-Opinions}). Note that the computation of $\bar{x}= (I - \EE(H(G)))^{-1}Bx(0)$ using sampling methods requires only inverting one final matrix, contrary to $\E(x(G,\infty))$ that requires one inversion operation per sample.

Our previous result follows as a corollary of a more general convergence theorem for random matrices that we prove in Section \ref{sec:Main}. Namely, we show that when $p$ is over the connectivity threshold and the agents' condescendence is uniformly bounded by  $\bar{\alpha}<1$, we have that 
$\|\EE(\phi(H(G)))-\phi(\EE(H(G)))\|_{\ast}\to 0$ as $n\to\infty$,  where  $\phi$ is any analytic matrix function with convergence radius strictly larger than $\bar{\alpha}$ and where $\|\cdot\|_{\ast}$ is the matrix norm induced by the $\ell_{\infty}$-norm in $\RR^n$ (Theorem \ref{thm:Main-Matrix}). In particular, we can take the function $g(X)=(I-X)^{-1}$ to recover the mean-field approximation result for the stable opinions since Theorem \ref{thm:Main-Matrix} guarantees that $\EE((I - H(G))^{-1}) \approx (I - \EE(H(G)))^{-1}$ as $n$ grows large.
The proof of Theorem \ref{thm:Main-Matrix} is based on providing, for each fixed $k\in\N$, upper bounds for the difference $\|\E(H^k) - \E(H)^k\|_{\ast}$ of the form $O(1/\log(n))$. This bound is derived by reducing our analysis to the study of the expected value of the random variable $(\deg_i(G) + k)^{-k}$, where $k<n$ and $\deg_i(G)$ is the degree of node $i$ in $G$.

In Section \ref{sec:RemarkDirected}, we consider the case of directed graphs and show that the convergence result $\|\E(x(G,\infty)) - \bar{x}\|_{\infty}\to 0$ also extends to the case where $p$ is over the connectivity threshold in the Erd\H{o}s-R\'enyi directed random graph model and the agents' condescendence is uniformly bounded by  $\bar{\alpha}<1$ (Theorem \ref{thm:Main-Opinions-Directed}). In this case, the proof follows by improving the upper bound of $\|\E(H^k) - \E(H)^k\|_{\ast}$ from $O(1/\log(n))$ to $O(1/n)$, which then readily extends Theorem \ref{thm:Main-Matrix}. Moreover, by paying a mild extra $\log(n)$ factor on the connectivity regime, we extend our results by also deriving $\|\E(x(G,\infty)) - \bar{x}\|_{\rho}\to 0$ for any $\rho\in (1,\infty)$ (Theorem \ref{thm:Main-Opinion-Extension}). Finally, in Section \ref{sec:remarks}, we discuss some open questions and possible extensions. 

\subsection{Related Work}
The model we study in this work is based on the contribution of Ghaderi and Srikant \cite{GhaderiSrikant2014Opinion}, where they analyze how the network structure, the intrinsic opinions, and the presence of stubborn agents influence the vector of stable opinions. Their approach is not limited to specific types of graphs but instead focuses on the matrix spectrum of the system, which characterizes the convergence time of the opinion update process.

Very recently, Xing and Johansson \cite{XingJohansson2024Concentration} studied concentration properties in the dynamics of gossip opinions in random graphs. Similar to our work, they also provide concentration for a process obtained by averaging over all the possible random graph realizations that depend on the size of the graph, considering networks in which agents do not have an intrinsic opinion, and there is a significant number of fully stubborn agents. However, their opinion update process differs from our model, and their analysis focuses on spectral properties of the matrices defining the dynamical system. Our work, on the other hand, focuses on partially stubborn agents, and our analysis is based on exploiting the combinatorial structure behind the Erd\H{o}s-Rényi model. 

The literature about opinion dynamics, contagion, and beliefs evolution in graphs is vast, and during the last decade, different models have been intensively studied from different perspectives, including incentives and opinion formation games \cite{Ferraioli2016Descentralized,Bindel2011HowBad,Acemoglu2011Learning,Acemoglu2011Bayesian,deVos2024Influencing}, susceptibility, polarization, and disagreement \cite{Parsegov2017Multidimensional, Fotakis2023Limited,Abebe2018VaryingSusceptibility,Bernardo2024survey,Mirtabatabaei2012Opinion,Amelkin2017Polar,Acemouglu2013Opinion,XingJohansson2024Concentration}, and the uncertainty modeling in the interaction through random graphs \cite{Munoz2025Exploring,Yang2024MinMax,Chakraborti2023Majority,Meier2017Push,Lopez-Pintado2006Contagion,Geethu2021Controllability}.
Finally, we mention that on a different line of the literature, we find the study of iterated random functions, which provide a way of analyzing random dynamical systems (see, e.g., Diaconis and Freedman \cite{DiaconisFreedman1999Iterated}). However, unlike our approach, the iterated function approach samples independent draws at each iteration. Instead, in our model, a single graph representing the network is randomly drawn, over which the iterations of the opinion process are later performed.
\section{Notation and Preliminaries on Graphs}\label{sec:Pre}


In what follows, we introduce some notation that will be used throughout this work. For a positive integer $n\in\N$, we write $[n] = \{1,\ldots,n\}$. For a finite set $S$, we denote its cardinality by $|S|$. For any integer $k\in\N$, we denote by $\mathcal{P}_k[\R]$ the space of polynomials, with real coefficients, of degree at most $k$. We denote the natural logarithm function by $\log(\cdot)$. 
For a square matrix $H = (h_{ij}) \in\R^{n\times n}$,
we consider the norm $\|\cdot\|_{\ast}$ given by 
\begin{equation}\label{eq:Norm-Star}
\left\| H \right\|_{\ast} = \max_{i\in [n]} \sum_{j = 1}^{n}|h_{ij}|.
\end{equation}
The norm $\left\| \cdot \right\|_{\ast}$ is induced by the $\ell_{\infty}$ norm in $\mathbb{R}^{n}$, that is, $\|H\|_{\ast} = \sup \{ \|Hx\|_{\infty}: \|x\|_{\infty}=1\}$. Therefore, it is also submultiplicative (see, e.g.,  \cite[Chapter 5]{HornJohnson2013Matrix}); that is, for every pair of matrices $H,B$, we have $\left\| HB \right\|_{\ast} \leq \left\| H \right\|_{\ast}\left\| B\right\|_{\ast}$.

A square matrix $H = (h_{ij})\in \R^{n\times n}$ is said to be \textit{substochastic} if it has nonnegative entries, for every $i\in [n]$ we have $\sum_{j=1}^n h_{ij}\leq 1$ and there exists $i_0\in[n]$ such that $\sum_{j=1}^n h_{i_0j}< 1.$
The matrix $H$ is said to be strictly substochastic if it is substochastic and $\sum_{j=1}^n h_{ij}< 1$ for every $i\in [n]$. It is well known that the power series of strictly substochastic matrices are convergent, and they verify the identity
\begin{equation}\label{eq:PowerSeriesIdentity}
(I-H)^{-1} = \sum_{k=0}^{\infty} H^k.
\end{equation}

\paragraph{Graphs notation.} Given a simple and undirected graph $G$ with nodes $\{1,\ldots,n\}$, we denote by $E(G)$ its set of edges. For a pair $i,j \in [n]$, we denote by $\{i,j\}$ the edge connecting $i$ with $j$.
For a node $i\in [n]$, we denote by $N_G(i) = \{j \in [n] : \{i,j\} \in E(G)\}$ the set of neighbours of $i$, and by $\deg(G,i) = |N_G(i)|$, its degree. We say a node $i$ is isolated if $\deg(G,i) = 0$. When the context is clear, we will omit the graph $G$ from the previous notations, writing simply $E$, $N(i)$, and $\deg(i)$.

A {\it trail} $c$ of size $k\in \N$ in the graph $G$ is an ordered sequence of nodes $c_1,c_2,\ldots,c_{k+1}$ in $[n]$, not necessarily all different.\footnote{We remark that trails are allowed to have consecutive repeated nodes.} We write $c = c_1c_2\cdots c_{k+1}$ and we denote the vertices of $c$ as $V(c)$. We denote by $V_{\ell}(c)$ the set of all nodes present in the first $\ell$ elements of $c$, that is, $V_{\ell}(c) =\{c_1,\ldots,c_{\ell}\}$.
We say that a trail is a {\it path} in $G$ if $\{ \{c_{r},c_{r+1}\}: r \in [k]\}\subset E(G)$. In such a case,  we denote the length (i.e., the number of edges) of $c$ as $|c|$. Note that for a path $c$ of size $k$, $|c|\leq k$ and the inequality might be strict if some node is repeated in the sequence. We use the notation $c_r$ to denote the $r$-th node in the trail $c$. We say a trail $c=c_1c_1\cdots c_{k+1}$ connects the nodes $i$ and $j$, if $c_1=i$ and $c_{k+1}=j$.

In what follows, for $n\in \N$, we write by $\mathbb{G}_n$ to denote the set of all graphs with nodes in $[n]$. Similarly, we write by $\mathcal{E}_n$ as the set of all possible edges for a graph in $\mathbb{G}_n$; note that $|\mathcal{E}_n| = n(n-1)/2$, and $\mathbb{G}_n$ can be identified with the vector space $\{0,1\}^{\mathcal{E}_n}$. 
We introduce the following notation for indicator functions, which is useful in the sequel:
\begin{enumerate}
    \item For $e \in \mathcal{E}_n$, we define $\ind_e: \mathbb{G}_{n}\to\{0,1\}$ as $\ind_e(G) = 1$ if $e \in E(G)$, and $\ind_e(G) = 0$ otherwise.
    \item For a trail $c = c_1c_2\cdots c_{k+1}$, we define $\ind_c: \mathbb{G}_{n}\to\{0,1\}$ as $\ind_c(G) = 1$ if $c$ is a path of $G$, and $\ind_c(G) = 0$ otherwise. Note that $\ind_c=\prod_{r=1}^{k} \ind_{c_rc_{r+1}}$, where to avoid notational clutter we are writing $\ind_{c_rc_{r+1}}$ instead of $\ind_{\{c_r,c_{r+1}\}}$, and we use the convention $\ind_{ii} \equiv 0$.
    \item For $K \subseteq \mathcal{E}_n$, we define $\ind_K: \mathbb{G}_{n}\to\{0,1\}$ as $\ind_K(G) = \prod_{e \in K} \ind_e(G)$.
\end{enumerate}

\subsection{Random Graphs Model}
We consider the random graph model introduced by Gilbert~\cite{Gilbert}, and Erd\H{o}s and R\'enyi~\cite{Erdos-Renyi}. They can be described as the random selection of a graph in $\mathbb{G}_n$, such that each edge connecting two elements exists, independently from the rest, with a fixed probability. This random graph model is one of the most fundamental in the literature, and we refer to primer monographs of graph theory for more details (see, e.g., \cite{Diestel2018Graph}). 
In this model, we are given a value $p\in (0,1)$ and $n\in \N$, 
and let $\PP$ be the probability measure such that
\begin{equation}\label{def:ERG-undirected}
    \PP(G) = \prod_{e \in E(G)} p \prod_{e \notin E(G)} (1-p)
\end{equation}
for every graph $G \in \mathbb{G}_{n}$. We denote the probability space $(\mathbb{G}_{n},\mathcal{P}(\mathbb{G}_{n}),\PP)$ by $\mathbb{G}(n,p)$.
In the following, we will work with a parameter $p = p(n)$, which might depend on the size $n$. In general, to save notation, we will omit the dependency on $n$ for $p$.

Note that in this construction, the measure $\PP$ can be viewed as the product measure over independent Bernoulli trials for each edge. That is, $\PP$ is the unique measure on $(\mathbb{G}_{n},\mathcal{P}(\mathbb{G}_{n}))$ that satisfies, for all $e \in \mathcal{E}_n$ we have
$\PP(e \in E(G)) = p.$
In particular, the random variables $\{\ind_e\, :\, e \in \mathcal{E}_n\}$ are mutually independent.
In what follows, for a set $K \subset \mathcal{E}_n$, we denote by $(\ind_e\ :\ e \in K)$ the random vector in $\{0,1\}^{|K|}$ given by the corresponding indicator functions. Although this notation is ambiguous as it does not specify a particular order, we will assume that $\mathcal{E}_n$ has an enumeration and that $(\ind_e\ :\ e \in K)$ follows the order induced by this enumeration. Using well-known facts about mutual independence of the composition between random variables and measurable functions (see, e.g., \cite[Chapter 3]{grimmett2020probability}), one can directly derive the following natural proposition that will be used several times in the sequel; we omit its proof as it is standard in the literature.

\begin{proposition}\label{prop:independienciaMutuaIndicatrices} 
Consider the Erd\H{o}s-Rényi model $\mathbb{G}(n,p)$, and let $K_1, \ldots, K_s$ be pairwise disjoint non-empty subsets of $\mathcal{E}_n$. Then, for any sequence of functions $\varphi_1, \ldots, \varphi_s$ with $\varphi_i: \{0,1\}^{|K_i|}\to \mathbb{R}$ for each $i\in [s]$, the random variables $X_1, \ldots, X_s$, where $X_i = \varphi_i((\ind_e\ : e \in K_i))$ for all $i \in [s]$, are mutually independent.
\end{proposition}

We finish this section by recalling a classic result of the Erd\H{o}s-R\'{e}nyi model about the threshold for the existence of isolated nodes, i.e., nodes with degree zero (see, e.g., \cite[Corollary 3.31]{LibroRG1}).
For a function $h:\N\to\R_+$, we write that $p\gg h$ when $p(n)/h(n)\to \infty$ as $n\to \infty$.

\begin{proposition}\label{Prop.NodoAislado}
If $p \gg \log(n)/n$, the probability that there exists an isolated node in a random graph $G\sim\mathbb{G}(n,p)$ goes to zero as $n\to \infty$.
\end{proposition}

\section{Opinion Dynamics and Main Results}\label{sec:Models}

In this work, an agent updates his opinion according to a weighted average of his initial intrinsic opinion and the opinion of every other neighboring agent in the network (Ghaderi and Srikant~\cite{GhaderiSrikant2014Opinion}). The more weight an agent gives to his opinion, the more {\it stubborn} the agent. 
The underlying network is typically modeled in the literature using a fixed graph. In contrast, we suppose the graph $G$ is random and drawn according to $\mathbb{G}(n,p)$. Each node $i\in [n]$ represents an agent of this network that has an initial opinion $x_i(G,0)=x_i(0) \in [0,1]$, where zero and one represent opposite opinions over some subject. 
Furthermore, for each agent $i\in [n]$, there is a value $\alpha_i\in [0,1]$ that captures the willingness of the agent to change the opinion.
We consider discrete-time stages, and at time $t+1$, the opinion $x_{i}(G,t+1)$ of each node $i\in [n]$ is given by
\begin{equation}\label{eq:OpinionModel}
    x_{i}(G,t+1) =
\begin{cases}
 \alpha_{i}\sum_{j \in N(i)} x_j(G,t)/\deg(i) + (1- \alpha_{i} )x_{i}(0) & \text{if } \deg(i)> 0, \\
 \alpha_{i}x_i(G,t) + (1 - \alpha_{i})x_i(0) & \text{if } \deg(i)=0. 
\end{cases}
\end{equation}
We define the {\it stable} opinion vector, if it exists, as
\begin{equation}\label{eq:LimitOpinion}
    x(G,\infty) = \lim_{t\to \infty} x(G,t).
\end{equation}

Note that in the opinion dynamics, we have that at each (discrete) time $t>0$, the opinion vector $x(G,t)$ is, in fact, a random vector, depending on the random graph $G$. We also remark that the equation defining the dynamic for isolated nodes is an artifact since it implies that whenever $i\in [n]$ is isolated, $x_i(G,t) = x_i(0)$ for every $t\in \N$. 
The value $1-\alpha_i$ measures the stubbornness of agent $i$. It is well-known that if all agents are {\it condescending} (i.e., $\alpha_i = 1$), the opinion dynamics reaches a consensus \cite{degroot1974reaching}. Our model assumes that all agents are partially stubborn, that is, $\alpha_i<1$ for each $i\in [n]$. Moreover, we assume that, regardless of the graph size $n$, there is a universal bound $\bar{\alpha}\in (0,1)$ for which all agents are at least $1-\bar{\alpha}$ stubborn.

\begin{assumption}\label{Assumption-UniformBound} There exists $\bar{\alpha}\in (0,1)$ such that for all $n\in \N$ and every $i\in [n]$, it holds $\alpha_i\leq \bar{\alpha}$.
\end{assumption}
\bigskip

Observe that the model given by \eqref{eq:OpinionModel} can be written compactly as
\begin{equation} \label{SLD}
    x(G,t+1) = H(G)x(G,t) + Bx(0),
\end{equation}
where $x(G,t) = [x_{1}(G,t), \cdots, x_{n}(G,t)]^{\top}$, $B$ is the diagonal matrix $\mathrm{diag}(1- \alpha_{1},\ldots,1-\alpha_n)$, and $H(G)$ is the matrix given by
\begin{equation*}
H_{ij}(G) =
    \begin{cases}
        \alpha_{i}A_{ij}(G)/\deg(i) & \text{if } \deg(i) > 0, \\
        \alpha_i\cdot 1_{\{i=j\}} & \text{if } \deg(i) = 0,  
    \end{cases}
\end{equation*}
with $A(G)$ being the adjacency matrix of the graph $G$, i.e., $A_{ij}(G)=1$ if $\{i,j\}\in E(G)$, and is zero otherwise. Under Assumption \ref{Assumption-UniformBound}, the matrix $H$ is strictly substocastic and $\|H\|_{\ast}\leq \bar{\alpha}$.
Then, following a simple computation using Equation \eqref{SLD}, we deduce that
\begin{equation}\label{EcuacionModeloMatricialD}
    x(G,t+1) = H(G)^{t}x(0) + \sum_{k = 0}^{t-1} H(G)^{k}Bx(0).
\end{equation}
Together with \eqref{eq:PowerSeriesIdentity}, the following proposition is directly implied.

\begin{proposition}\label{prop:LimitOpinion}
      Under Assumption \ref{Assumption-UniformBound}, the stable opinion $x(G,\infty)$ exists for every $G\in \mathbb{G}_n$, and is given by 
    \begin{equation*}
        x(G,\infty) =\lim_{t \to \infty}x(G,t)=  (I-H(G))^{-1}Bx(0).
    \end{equation*}
    In particular, when $G\sim \mathbb{G}(n,p)$, we have $\E(x(G,\infty)) = \E((I-H(G))^{-1})Bx(0)$.
\end{proposition}


\subsection{Mean-Field Dynamics}

While Proposition~\ref{prop:LimitOpinion} provides a closed form for the stable opinion, this expression is graph-dependent, and sampling approaches to computing $\E(x(G,\infty))$ might become untractable as the size of the network grows large. Then, it is natural to study a deterministic counterpart, which we call the {\it mean-field model}, given by the dynamics where we consider the expected matrix $\EE(H(G))$ instead of the random matrix $H(G)$. In this case, the dynamics are no longer random, and the initial opinions entirely determine the evolution of the opinions.
\begin{definition}\label{def:MeanFieldModel}
Given a graph $G\in \mathbb{G}_n$, we have $\bar{x}(0)=x(0)$ and for each $t\in \N$ we have
    \begin{equation}
        \bar{x}(t+1) = \EE(H(G))^{t}x(0) + \displaystyle \sum_{k = 0}^{t-1}\EE(H(G))^{k}Bx(0).\label{MFD}
    \end{equation}
We define the stable solution of the mean-field dynamics, if it exists, as
\begin{equation}\label{eq:MeanFieldLimitOpinion}
    \bar{x}(\infty) = \lim_{t\to \infty} \bar{x}(t).\notag
\end{equation}
\end{definition}
Under Assumption \eqref{Assumption-UniformBound}, the matrix $\E(H(G))$ is strictly substochastic and $\|\E(H)\|_{\ast}\leq \bar{\alpha}$. Then, we can derive the following direct proposition, ensuring the stable opinion vector exists for the mean-field dynamics.
\begin{proposition}\label{prop:limit-meanfield}
    Under Assumption \ref{Assumption-UniformBound}, the stable solution $\bar{x}(\infty)$ exists and is given by 
    \begin{equation*}
        \bar{x}(\infty) = (I - \EE(H(G)))^{-1}Bx(0).
    \end{equation*}
\end{proposition}


\subsection{Formal Statement of Our Approximation Result}
In this section, we provide the formal statements of our main results. We show that the stable solution in the mean-field opinion dynamics \eqref{MFD} approximates the expected stable opinion in the original dynamics \eqref{eq:OpinionModel}. 
Observe that by Propositions \ref{prop:LimitOpinion} and \ref{prop:limit-meanfield}, and noting that $\|x(0)\|_{\infty} \leq 1$, one can see that 
\begin{align*}
    \| \E\left(x(G,\infty)\right) - \bar{x}(\infty) \|_{\infty}\leq \| \EE((I - H(G))^{-1}) - (I - \EE(H(G))^{-1}) \|_{\ast},
\end{align*}
Therefore, to show that $\E\left(x(G,\infty)\right)\approx \bar{x}(\infty)$ as $n$ grows large, it is sufficient to study whether $\EE((I - H(G))^{-1}) \approx (I - \EE(H(G)))^{-1}$ as $n$ grows large. In our first main technical result (Theorem \ref{thm:Main-Matrix} in Section~\ref{sec:Main}), we show that such approximation holds in a quite general setting as long as $p\gg \log(n)/n$, namely, we have 
\[
\EE(\phi(H(G))) \approx \phi(\EE(H(G)))
\] 
for any matrix function $\phi$ that admits a power series expansion where $\bar{\alpha}$ lies within the convergence radius for the $\|\cdot\|_{\ast}$ norm. In particular, we can take the function $\phi(X)=(I-X)^{-1}$ to recover the mean-field approximation result for the stable opinions.
We prove Theorem \ref{thm:Main-Matrix} in Section \ref{sec:Main}. 
We get the following theorem as a direct corollary.
\begin{theorem}\label{thm:Main-Opinions}
    Let $G\sim \mathbb{G}(n,p)$. Suppose that Assumption \ref{Assumption-UniformBound} holds and that $p\gg \log(n)/n$. Then, 
    \begin{equation*}
       \lim_{n \to \infty} \| \E\left(x(G,\infty)\right) - \bar{x}(\infty) \|_{\infty} = 0.
    \end{equation*}
\end{theorem}
In Section \ref{sec:RemarkDirected}, we introduce the opinion dynamics in the directed-graph model, and we study extensions of our approximation result to this setting.
\section{Mean-Field Approximation for General Matrix Functions}\label{sec:Main}

We devote this section to prove our main technical result regarding the mean-field approximation of matrix functions obtained from analytic functions over the reals. Formally, we say that a function $\phi:\RR\to \RR$ is {\it analytic with convergence radius} $\rho$ if there exists a sequence $(b_k)_{k\in \NN}$ such that 
$\phi(x)=\sum_{k=0}^{\infty}b_kx^k$
for every $x\in \RR$ such that $|x|< \rho$. For any such function $\phi$, one can consider a matrix function given by
$$\phi(X)=\sum_{k=0}^{\infty}b_kX^k$$
where $X$ is a $n\times n$ matrix. It is well-known that if $\phi$ is analytic with convergence radius $\rho$, then the matrix power series counterpart also converges as long as $\|X\|_{\ast}<\rho$. In fact, this holds for any submultiplicative matrix norm, then any induced norm (see, e.g., \cite[Theorem 5.6.15]{HornJohnson2013Matrix}); note that in this reference, every matrix norm is assumed to be submultiplicative \cite[Chapter 5.6]{HornJohnson2013Matrix}. The following is the main result of this section.
\begin{theorem}\label{thm:Main-Matrix} 
Consider an analytic matrix function $\phi$ with convergence radius $\rho$, and suppose that Assumption \ref{Assumption-UniformBound} holds with $\bar{\alpha}<\rho$, and $p\gg \log(n)/n$. Then, when $G\sim \mathbb{G}(n,p)$, we have
    \begin{equation}\label{eq:Main-matrixEquation}
        \lim_{n\to\infty} \|\EE(\phi(H(G))) - \phi(\EE(H(G))) \|_{\ast} = 0.\notag
    \end{equation}
\end{theorem}
In particular, the matrix function $X\mapsto (I-X)^{-1}$ is analytic with convergence radius equal to one, and therefore we can apply Theorem \ref{thm:Main-Matrix} to conclude the mean-field stable opinion approximation in Theorem \ref{thm:Main-Opinions}.
The rest of this section is devoted to prove Theorem \ref{thm:Main-Matrix}. 
Note that
\begin{align}
    \|\EE(\phi(H(G))) - \phi(\EE(H(G))) \|_{\ast} &= \left\|\EE\left(\lim_{t\to \infty}\sum_{k=0}^{t-1} b_kH(G)^{k}\right) - \lim_{t \to \infty} \sum_{k = 0}^{t-1} b_k\EE(H(G))^{k} \right\|_{\ast}\nonumber \\
    & =\left\| \lim_{t\to \infty}\sum_{k=0}^{t-1} b_k \EE(H(G)^{k}) - \lim_{t \to \infty} \sum_{k = 0}^{t-1} b_k\EE(H(G))^{k}\right\|_{\ast} \nonumber \\
    &\leq  \limsup_{t\to \infty} \sum_{k=0}^{t-1} b_k\|\EE(H(G)^{k})-\EE(H(G))^{k}\|_{\ast}, \label{desig.sum||E(H^k) - E(H)^k||*_cap3}
\end{align}
where the first equality follows by the dominated convergence theorem (see, e.g., \cite[Chapter 2]{folland1999real}). With this observation, we reduce the problem to studying the difference between $\EE(H(G)^{k})$ and $\EE(H(G))^{k}$ in order to apply a uniform bound argument on \eqref{desig.sum||E(H^k) - E(H)^k||*_cap3}. In what follows, for $i,j\in [n]$ and $k\in\N$ we write $H^k_{ij}$ and $\EE(H)^k_{ij}$ to denote the $(i,j)$-entry of the power matrices $H^k$ and $\EE(H)^k$, respectively.

Recall that for $H(G)$, its entry $(i,j)$ is $\alpha_{i}/\deg(i)$ when $\deg(i) > 0 \text{ and }\{i,j\}\in E(G)$, $\alpha_i$ when $\deg(i) = 0 \text{ and }i=j$, and zero otherwise. We remark that while the entry $(i,i)$ in the adjacency matrix is zero for an isolated node, this is not the case in the matrix $H(G)$ as we still account for the stubbornness factor $\alpha_i$. To lighten the notation during this section, we will omit the dependency on the graph $G$ for matrices and graph parameters as the context is clear. For $G\sim \mathbb{G}(n,p)$ and $i,j\in [n]$, $H^k_{ij}$ is obtained by considering all the possible trails of size $k$ connecting $i$ and $j$. Namely,
\begin{equation}
    H^k_{ij} = \sum_{c\in C_{ij}(k)} \prod_{r=1}^{k} H_{c_rc_{r+1}},
\end{equation}
where $C_{ij}(k)$ is the set of trails with size $k$ connecting $i$ and $j$ in the complete graph with nodes $[n]$.
For a trail with no isolated nodes, the above expression reduces to 
\begin{equation}
    H^k_{ij} = \sum_{c\in C_{ij}(k)} \prod_{r=1}^{k} H_{c_rc_{r+1}} = \sum_{c\in C_{ij}(k)} \prod_{r=1}^{k} \alpha_{c_r}\frac{1}{\deg(c_r)}\ind_{c_r,c_{r+1}}.
\end{equation}
Then, for a trail $c\in C_{ij}(k)$, we study $\E(H_c)= \E(\prod_{r=1}^{k}H_{c_rc_{r+1}})$ and $\E(H)_c = \prod_{r=1}^{k}\E\left(H_{c_rc_{r+1}}\right)$. Our strategy consists in dividing the cases where $c$ contains isolated nodes or not, and then to use a truncated version of the function $i\mapsto 1/\deg(i)$, given by
\begin{equation}\label{funciónfiGND}
    \td_i = \begin{cases}
        1/(\deg_c(i)+\deg_{\bar{c}}(i)) & \text{ if } \deg(i) = \deg_c(i) +\deg_{\bar{c}}(i)  >0, \\
        0 & \text{ if } \deg(i) = 0,
    \end{cases}
\end{equation}
where $\deg_c(i)$ is the number of neighbors of $i$ in the trail $c$, and $\deg_{\bar{c}}(i)$ is the number of neighbors of $i$ that are not involved in the trail.  That is,
\[
\deg_c(i) = |N(i)\cap V(c)|\quad\text{and}\quad \deg_{\bar{c}}(i) = |N(i)\setminus V(c)|.
\]
A simple strategy might be thinking that, whenever $c$ was a trail without repeated nodes in their first $k$ nodes, we could have mutual independence and deduce that $|\E(H_c) - \E(H)_c| = 0$. If such analysis would hold, then our study would be reduced to look into trails with repeated nodes. However, this is not true since the random variables $\{\td_i\,:\, i\in[n]\}$ are not mutually independent in general.  Indeed, for $G\sim \mathbb{G}(n,p)$, the variables $\td_i$ and $\td_j$ are linked to the random variable $\ind_{ij}$.
The key point is that whenever $i,j\in V(c)$, their mutual dependency is encompassed within $\deg_c(i)$ and $\deg_c(j)$, while the variables $\deg_{\bar{c}}(i)$ and $\deg_{\bar{c}}(j)$ remain independent. Our analysis heavily relies on this observation, and so, we start by studying upper bounds for the term $\E\left( \td_{i}\ind_{ij}\right)$.

\begin{lemma}\label{lemma:Upper-Bound-fi-OneEdge}
Let $G \sim \mathbb{G}(n,p)$ and let $i, j \in [n]$ be any pair of different nodes in $G$. Then,
\begin{align*}
    \E(\td_{i}|\ind_{ij}=1) \leq \frac{1}{(n-1)p}.
\end{align*}
\end{lemma}

\begin{proof}
Given $\ind_{ij} = 1$, i.e., the edge $\{i,j\}$ is in the graph $G$, we ensure that $\td_{i}  = 1/\deg(i)$. Therefore, 
\begin{align*}
    \EE(\td_{i}|\ind_{ij}=1) &= \sum_{k=0}^{n-2}\frac{1}{k+1} \binom{n-2}{k}p^{k}(1-p)^{n-2-k} \nonumber \\
                      &= \frac{1}{(n-1)p} \sum_{k=0}^{n-2} \binom{n-1}{k+1}p^{k+1}(1-p)^{(n-1)-(k+1)}\nonumber \\
                      &= \frac{1}{(n-1)p} \sum_{k=1}^{n-1} \binom{n-1}{k}p^{k}(1-p)^{(n-1)-k}  \\
                      &\leq \frac{1}{(n-1)p}(p+(1-p))^{n-1} = \frac{1}{(n-1)p}. \qedhere
\end{align*}
\end{proof}


We need to provide a similar bound to study the trails with repeated nodes, which is summarized in the following lemma.
\begin{lemma}\label{lemma:Upper-Bound-fi-MultipleEdges}
Let $G \sim \mathbb{G}(n,p)$ and $k<n$. Let $c$ be a trail of size $k$ and let $V_k(c) = \{i_1,\ldots,i_s\}$. For each $r\in [s]$ let $m_r$ be the number of repetitions of $i_r$ within the first $k$ elements of $c$. Then,
\begin{align*}
\EE\left(\left.\td_{i_1}^{m_{1}} \cdots \td_{i_{s}}^{m_{s}}\right| \ind_c = 1 \right)  
 \leq \prod_{r=1}^s \frac{m_r!}{p^{m_r}(n-k)^{m_r}} \leq \frac{k!}{p^k(n-k)^k}.
\end{align*}
\end{lemma}

\begin{proof}
First, note that $\ind_c = 1$ yields that the trail does not have two equal consecutive nodes, and second, the variables $\{\ind_c,\deg_{\bar{c}}(i_1),\ldots,\deg_{\bar{c}}(i_s)\}$ are mutually independent by Proposition \ref{prop:independienciaMutuaIndicatrices}. Thus, we have
\begin{align}
    \EE\left(\td_{i_1}^{m_{1}} \cdots \td_{{i}_s}^{m_{s}}\left|\ind_c\right. = 1 \right) &\leq \EE\left( \prod_{r=1}^{s}\left(\left.\frac{1}{1 + \deg_{\bar{c}}(i_r)}\right)^{m_{r}} \right| \ind_c = 1 \right)\nonumber   \\
&=\prod_{r=1}^{s} \E\left( \left(\frac{1}{1 + \deg_{\bar{c}}(i_r)}\right)^{m_{r}}\right). \label{mulEspe}
\end{align}
For a fixed $r\in [s]$, we analyze the term $\theta_r:=\E((1/(1+\deg_{\bar{c}}(i_r)))^{m_{r}})$. Suppose that $m_r\geq 2$. Recall the following two inequalities hold for every $a\in\N$:
Firstly, ${(t+a)}/{(t+1)} = 1 + {(a-1)}/{(t+1)} \leq a$ for all $a \geq 2$, and secondly $(n - a)^{a} \leq \prod_{l=1}^{a}(n - a)$. Let $l=|V(c)|$, which is either $s$ or $s+1$ depending if $c_{k+1}\in V_k(c)$ or not. Then, we have
\begin{align} 
    \theta_r&= \sum_{t = 0}^{n-l} \left(\frac{1}{1+t}\right)^{m_r} \binom{n-l}{t}p^{t}(1-p)^{n-l-t}\nonumber \\
    &= \sum_{t = 0}^{n-l}\left(\prod_{a=2}^{m_r}\frac{a+t}{1+t}\right)\frac{(n-l)!}{(t+m_r)!(n-l-t)!}p^{t}(1-p)^{n-l-t}\nonumber \\
    &\leq  m_r!\sum_{t = 0}^{n-l}\frac{(n-l)!}{(t+m_r)!(n-l-t)!}p^{t}(1-p)^{n-l-t}\nonumber \\
    &\leq \frac{m_r!}{(n-l)^{m_r}}\sum_{t = 0}^{n-l}\frac{(n+m_r-l)!}{(t+m_r)!(n +m_r - l-(t+m_r))!}p^{t}(1-p)^{n +m_r -l-(t+m_r)}\nonumber \\
    &= \frac{m_r!}{(n-l)^{m_r}}\sum_{t = 0}^{n-l}\binom{n+m_r-l}{t+m_r}p^{t}(1-p)^{n +m_r -l-(t+m_r)}\nonumber \\
    &= \frac{m_r!}{p^{m_r}(n-l)^{m_r}}\sum_{t = m_r}^{n+m_r-l}\binom{n+m_r-l}{t}p^{t}(1-p)^{n +m_r -l-t}\nonumber \\
   &\leq \frac{m_r!}{p^{m_r}(n-k)^{m_r}}\nonumber,
\end{align}
where the first equality holds since $(a+t)/(1+t)=1$ for $a=1$, and the last inequality follows due to $l\leq k$, which is guaranteed by the fact that there is at least one repeated node when $m_r\geq 2$. Now, note that the inequality also holds for $m_r=1$ by Lemma \ref{lemma:Upper-Bound-fi-OneEdge}. Returning to equation \eqref{mulEspe}, we deduce that
\begin{align}
    \prod_{r=1}^{s} \E\left( \left(\frac{1}{1 + \deg_{\bar{c}}(i_r)}\right)^{m_r}\right) &\leq \prod_{r=1}^{s}\frac{m_r!}{(n-k)^{m_r}p^{m_r}} \nonumber \\
    &=  \frac{m_1!\cdots m_s!}{p^{m_{1}+\cdots +m_{s}}(n-k)^{m_{1}}\cdots (n-k)^{m_{s}}} \nonumber \\
    & \leq \frac{k!}{p^{k}(n-k)^k}, \nonumber
\end{align}
where the last inequality considers that $m_{1} + \cdots + m_{s} = k$. The proof is then completed.
\end{proof}


As we will see in the sequel, we also need lower bounds to study trails without repetitions.

\begin{lemma}\label{lemma:Lower-Bound-fi-MultipleEdges}
When $n$ is sufficiently large, $p\gg\log(n)/n$, and $G\sim\mathbb{G}(n,p)$, for every $i,j \in [n]$ we have
\begin{equation}\label{eq:One-Edge-lowerbound-UD}
    \EE(\td_{i}|\ind_{ij} = 1) \geq \frac{1}{np}.\notag
\end{equation}
\end{lemma}
\begin{proof}
  Following the exact same development as the proof of Lemma \ref{lemma:Upper-Bound-fi-OneEdge}, we have that
\begin{equation}\label{Ultima_Igualdad_Previa}
    \EE(\td_{i}|\ind_{ij}=1) =\frac{1}{(n-1)p} \sum_{k = 1}^{n-1}\binom{n-1}{k}p^k(1-p)^{(n-1)-k}.\notag
\end{equation}
By completing the binomial expansion of the right-hand side, we note that

\begin{equation*}
    \displaystyle\frac{1}{(n-1)p} \sum_{k = 1}^{n-1}\binom{n-1}{k}p^k(1-p)^{(n-1)-k}= \frac{1}{(n-1)p}\left(1 - (1-p)^{n-1}\right).
\end{equation*}
Since $p \gg\log(n)/n$, we have 
\[1 - (1-p)^{n-1}\geq 1 - \frac{1}{n},\]
for sufficiently large $n$. Therefore, it follows that for all sufficiently large $n$,

\begin{equation*}
    \EE(\td_{i}|\ind_{ij}=1) = \frac{1 - (1-p)^{n-1}}{(n-1)p}  \geq  \frac{(1-\frac{1}{n})}{(n-1)p} = \frac{1}{np}.\qedhere
\end{equation*}  
\end{proof}


An alternative proof of Lemma~\ref{lemma:Lower-Bound-fi-MultipleEdges} is to observe that if we denote by $V'\subseteq [n]\setminus\{j\}$ the edges incident to $i$ apart from $\{i,j\}$, then $\EE(\td_{i}|\ind_{ij}=1) = \E\left({1}/{(\deg_{V'}(i) + 1)}\right)$, and then use the identity
\[
\E\left(\frac{1}{\deg_{V'}(i) + 1}\right) = \frac{1 - (1-p)^{n-1}}{(n-1)p},
\]
which is a well-known identity for binomial distributions (see, e.g.,~\cite{ChaoStrawderman1972Negative}). In general, it is shown that whenever $X$ is a random variable following a binomial distribution of parameters $\bar{n}$ and $p$, the following holds:
\begin{equation}\label{eq:ChaoStrawderman-original}
\begin{aligned}
    \E\left(\frac{1}{X+k}\right)
     &=q^{\bar{n}}\left(\frac{q}{p}\right)^k\left[ \left(\sum_{s=1}^{k-1}(-1)^{s+1}\frac{(k-1)(k-2)\cdots (k-s+1)}{(\bar{n}+1)(\bar{n}+2)\cdots(\bar{n}+s)}\left(\frac{p}{q}\right)^{k-s}\frac{1}{q^{\bar{n}+s}}\right) \right.\\
    & \quad\quad \left.+\; \frac{(-1)^{k-1}(k-1)!}{(\bar{n}+1)(\bar{n}+2)\cdots(\bar{n}+k)}\left(\frac{1}{q^{\bar{n}+k}} - 1\right)\right],
\end{aligned}
\end{equation}
where $q=1-p$, and the convention $(k-1)(k-2)\cdots (k-s+1) = 1$ for $s=1$ is considered. We will use this formula to derive the following last lower bound for the particular case of a trail of size $k$  without repetitions on its first $k$ nodes.


\begin{lemma}\label{lemma:Lower-Bound-fi-PathsWithoutRep}
   Let $G\sim\mathbb{G}(n,p)$, and let $c$ be a trail in $G$ of size $k$ without repetitions on its first $k$ nodes. Then, for sufficiently large $n$, we have
    \begin{equation*}
        \EE\left.\left(\prod_{r = 1}^{k} \td_{c_r}\ind_{c_r c_{r+1}}\,\right|\, \ind_c = 1\right) \geq  \frac{1}{(n-k)^{k}p^k} \left(1 - \sum_{s=1}^{k-1}\frac{(k-1)!}{(n-k)^sp^s}\right)^k.
    \end{equation*}
When $p\gg \log(n)/n$, we have
\begin{equation}\label{eq:PathNoRep-LowerBound}
        \EE\left.\left(\prod_{r = 1}^{k} \td_{c_r}\ind_{c_r c_{r+1}}\,\right|\, \ind_c = 1\right) \geq  \frac{1}{(n-k)^{k}p^k} \left(1 - \frac{1}{\log(n)}\right)^k.
    \end{equation}
\end{lemma}
\begin{proof}
Note that, since $|c|=k$, we have that $\deg_c(c_r)\leq k$ for each $r$. Then,
\begin{align*}
    \E\left.\left(\prod_{r=1}^{k}\td_{c_r}\right|\ind_c=1\right) & = \E\left.\left(\prod_{r=1}^{k}\frac{1}{\deg_c(c_r)+\deg_{\bar{c}}(c_r)}\right|\ind_c=1\right) \\
    &\geq \E\left.\left(\prod_{r=1}^{k}\frac{1}{k+\deg_{\bar{c}}(c_r)}\right|\ind_c=1\right) = \prod_{r=1}^{k} \E\left(\frac{1}{k+\deg_{\bar{c}}(c_r)}\right).
\end{align*}
The last equality  is due to the mutual independence between $\deg_{\bar{c}}(c_r)$ and $\ind_{c}$  (see Proposition \ref{prop:independienciaMutuaIndicatrices}). Now, let us fix $r\in [k]$. Set $\bar{n} = n-k-1$ and note that $\deg_{\bar{c}}(c_r)$ follows a binomial distribution with parameters $\bar{n}$ and $p$. Thus, equation \eqref{eq:ChaoStrawderman-original} applies, which can be rewritten as
\[
 \E\left(\frac{1}{k+\deg_{\bar{c}}(c_r)}\right)= \sum_{s=1}^{k-1} \frac{(-1)^{s+1}}{kp^s}\prod_{h = 1}^{s} \frac{k-h+1}{\bar{n} + h} + \frac{(-1)^{k-1}(k-1)!}{(\bar{n}+1)\cdots(\bar{n}+k)}\frac{(1-(1-p)^{\bar{n}+k})}{p^k}.   
\]
Then, by considering all terms but the first one ($s=1)$ as negative and noting that $1-(1-p)^{\bar{n}+k}\leq 1$, we can write
\begin{align*}
 \E\left(\frac{1}{k+\deg_{\bar{c}}(c_r)}\right)&\geq \frac{1}{(\bar{n}+1)p} - \sum_{s=2}^k  \frac{\prod_{h = 1}^{s} (k-h+1)}{k(\bar{n}+1)^sp^s} \\
 &\geq \frac{1}{(\bar{n}+1)p} - \sum_{s=2}^k  \frac{(k-1)!}{(\bar{n}+1)^sp^s} = \frac{1}{(n-k)p}\left(1 - \sum_{s=1}^{k-1}\frac{(k-1)!}{(n-k)^sp^s}\right).
\end{align*}
Then, the first inequality in the lemma statement follows. Now, if $p\gg \log(n)/n$, there is $n$ large enough such that $p\geq M\log(n)/(n-k)$ with $M= k!$. Then, we can write
\begin{align*}
    \frac{1}{(n-k)p}\left(1 - \sum_{s=1}^{k-1}\frac{(k-1)!}{(n-k)^sp^s}\right) &\geq \frac{1}{(n-k)p}\left(1 - \sum_{s=1}^{k-1}\frac{1}{kM^{s-1}\log(n)^s}\right)\\
    &\geq \frac{1}{(n-k)p}\left(1 - \sum_{s=1}^{k-1}\frac{1}{k\log(n)}\right)
    \geq \frac{1}{(n-k)p}\left(1 - \frac{1}{\log(n)}\right).
\end{align*}
The second inequality in the statement follows, so the proof is completed.
\end{proof}


We are now ready to present the main proposition of this section, which provides the desired convergence result for the terms $|\EE(H_{ij}^{k}) - \EE(H)^{k}_{ij}|$. Recall that  $H^k_{ij}$ and $\EE(H)^k_{ij}$ denote the entry $(i,j)$ of the power matrices $H^k$ and $\EE(H)^k$, respectively.

\begin{proposition}\label{prop:DifferenceForFixedExponent}
Let $G \sim \mathbb{G}(n,p)$ and suppose that $p\gg \log(n)/n$. Then, for every $k\in \N$, there exists a constant $C(k) > 0$ such that, for any sufficiently large $n\in \N$, and every $i,j\in [n]$, we have 
\begin{equation}\label{eq:Bound-by-coordinates}
    |\EE(H_{ij}^{k}) - \EE(H)^{k}_{ij}| \leq \frac{C(k)}{n\log(n)}.
\end{equation}
In particular, 
\(
 \|\EE(H^{k}) - \EE(H)^{k} \|_{\ast}\leq \frac{C(k)}{\log(n)} \to 0
\) as ${n\to\infty}$.
\end{proposition}
\begin{proof}
    Let $P_{ij}(k)$ be the set of all trails in  $C_{ij}(k)$ without repetitions on its first $k$ nodes, and let $R_{ij}(k)$ be the set of all trails in  $C_{ij}(k)$ with at least one repetition on its first $k$ nodes. Clearly, $C_{ij}(k) = P_{ij}(k) \cup R_{ij}(k)$. Let
\(
H_c = \prod_{r=1}^{k} H_{c_rc_{r+1}}
\)
and
\(
\EE(H)_c = \prod_{r=1}^{k} \E(H_{c_rc_{r+1}}).
\)
We can write 
\begin{align}
    |\EE(H_{ij}^{k}) - \EE(H)_{ij}^{k}| &= \textstyle\left|\sum_{c \in C_{ij}(k)}(\EE(H_{c}) - \EE(H)_{c})\right|\nonumber\\
    &\leq \textstyle\sum_{c \in R_{ij}(k)}|\EE(H_{c}) - \EE(H)_{c}| + \sum_{c \in P_{ij}(k)} |\EE(H_{c}) - \EE(H)_{c}|.\notag\label{Desigualdad_principal_Lema4.5}
\end{align}
Fix $c\in R_{ij}(k)$, and take $r\in[k]$. Suppose $c_r\neq c_{r+1}$. Note that if $\ind_{c_{r},c_{r+1}}= 0$, then $H_{c_r,c_{r+1}}= 0$ as well. 
Thus, using Lemma \ref{lemma:Upper-Bound-fi-OneEdge}, we have
 \begin{align*}
     \E(H_{c_rc_{r+1}}) &= \E(H_{c_rc_{r+1}} | \ind_{c_r,c_{r+1}} = 1)\ \PP(\ind_{c_r,c_{r+1}} = 1)\\
     &= \alpha_{c_r}\E(\td_{c_rc_{r+1}} | \ind_{c_r,c_{r+1}} = 1)\ p\leq \frac{1}{n-1}.
 \end{align*}
If $c_r=c_{r+1} = v$, we have $H_{vv} > 0$ if and only if $\deg(v)=0$. Thus, in this case, we have that for $n$ large enough
 \begin{align*}
     \E(H_{c_rc_{r+1}}) &= \E(H_{c_rc_{r+1}} | \ind_{c_r,c_{r+1}} = 0)\PP(\ind_{c_r,c_{r+1}} = 0)\\
     &= \alpha_{c_r}(1-p)^{n-1} \leq \frac{1}{n-1},
 \end{align*}
 where the last inequality holds since $p\gg \log(n)/n$. In both cases we have $ \E(H_{c_rc_{r+1}})\leq {1}/{(n-1)}$ and so
 \[
 \E(H)_c = \prod_{r=1}^k \E(H_{c_r,c_{r+1}}) \leq \frac{1}{(n-1)^k}.
 \]
In the following, we estimate $\E(H_c)$. We distinguish three cases:

\begin{enumerate}[label=\normalfont(\alph*)]
\item If $c$ has two nodes that are different and two consecutive nodes equal, then $\E(H_c) = 0$. Indeed, if such situation arises, one can easily deduce that there are three nodes $c_{r}$, $c_{r+1}$ and $c_{r+2}$ such that either $c_r = u$ and $c_{r+1} = c_{r+2} = v$ or $c_{r} = c_{r+1} = u$ and $c_{r+2} = v$. Then, for $H_{c}> 0$, one needs to have at the same time $\ind_{\{u,v\}} = 1$ and $\deg(v) = 0$, which is not possible.
\item If all nodes of $c$ are the same (that is, $c_{r} = i$ for all $r\in [k+1]$), then $\E(H_c) = \alpha_{i}^k(1-p)^{n-1}$. But this is just one case that we can treat separately. 
\item Suppose that $c$ has no consecutive repeated nodes. Let $s=V_k(c)$, let us write $V_k(c) = \{i_r \,:\, r\in[s]\}$ and that for every $r\in [s]$, let $m_r>0$ be the number of repetitions of $i_r$ in the first $k$ elements of $c$. Let  $K_r = \{ \{i_r,v\} :\{i_r,v\}\in E(c) \}$. Then, we have that 
\[
H_c = \prod_{r=1}^{s} (\alpha_{i_r}\td_{{i}_r})^{m_{r}}\ind_{K_{r}} \leq \left(\prod_{r=1}^{s}\td_{i_r}^{m_r}\right)\ind_c.
\]
Since $m_{i_1} + \cdots + m_{i_s} = k$, using Lemma \ref{lemma:Upper-Bound-fi-MultipleEdges}, we have that $\E(H_c| \ind_c = 1) 
    \leq k!/p^k(n-k)^{k}$. 
\end{enumerate}

Let $R^*_{ij}(k)$ be the set of all trails in $R_{ij}(k)$ without two equal consecutive nodes. By taking all the cases into account, we deduce that
\begin{align*}
\sum_{c \in R_{ij}(k)}|\EE(H_{c}) - \EE(H)_{c}| &\leq \sum_{c \in R_{ij}(k)} \EE(H)_c + \sum_{c\in R_{ij}(k)}\EE(H_c) \\
&= \frac{|R_{ij}(k)|}{(n-1)^k} + \overbrace{0}^{(a)} + \overbrace{\alpha_{i}^k(1-p)^{n-1}}^{(b)}+ \overbrace{\sum_{c\in R^*_{ij}(k)} \EE(H_c) }^{(c)}\\
&\leq \frac{|R_{ij}(k)|}{(n-1)^k} +  (1-p)^{n-1} +  \sum_{c\in R^*_{ij}(k)} \EE(H_c) ,
\end{align*}

We first need to estimate $|R_{ij}(k)|$. Recall that every $c\in R_{ij}(k)$ has $c_1 = i$ and $c_{k+1}=j$, and has at least one repetition among the nodes $c_1,\ldots,c_{k}$. Then, the following holds:
\begin{enumerate}[label=\normalfont(\alph*)]
    \item The number of trails in $R_{ij}(k)$ that have node $i$ repeated are at most $(k-1)n^{k-2}$. Indeed, there are $k-1$ possibilities for the position of the second appearance, and then at most $n^{k-2}$ combinations for the remaining positions, since the position $k+1$ is forced to be $j$.
    \item The number of trails in $R_{ij}(k)$ that have a node repeated different from $i$ are at most $\binom{k-1}{2} n^{k-2}$. Indeed, there are $\binom{k-1}{2}$ possible choices of two positions with a repeated node among the positions $2$ to $k$, $n$ possibilities for the repeated node, and $n^{k-3}$ possibilities for the remaining $k-3$ positions, since positions $1$ and $k+1$ are fixed.
\end{enumerate}
Therefore, with this analysis we have that
\begin{align}
    \frac{|R_{ij}(k)|}{(n-1)^k} 
    & \leq \frac{1}{(n-k)^{k}}\left((k-1) n^{k-2} + \binom{k-1}{2}n^{k-2}\right)\nonumber\\
    & = \frac{(k-1)(k+2)}{2}\cdot \frac{n^{k-2}}{(n-k)^{k}}\nonumber\\
    &= \frac{(k-1)(k+2)}{2(n-1)^{2}(1-1/n)^{k-2}}.\nonumber
\end{align}

Now, let us study $\sum_{c\in R^*_{ij}(k)} \EE(H_c)$. For $k=2$ one has that $R_{ij}^*(2) = \emptyset$. For $k\geq 3$, one has that for $c\in \R_{ij}^*(k)$, then $2\leq|V_k(c)|\leq k-1$. This is because at least one node needs to be repeated and since $c_1 = i$, then $c_2\neq i$. Let us denote
\begin{align*}
R_{ij}^*(k,s,0) &=\{ c\in R_{ij}^*(k)\ :\ |V_k(c)| = s, j\notin V_k(c) \},\\
R_{ij}^*(k,s,1) &=\{ c\in R_{ij}^*(k)\ :\ |V_k(c)| = s, j\in V_k(c) \}.
\end{align*}
Clearly, $R_{ij}^*(k) = \bigcup_{s=2}^{k-1}(R_{ij}^*(k,s,0)\cup R_{ij}^*(k,s,1))$. Moreover, one has that $|R_{ij}^*(k,s,0)| \leq  \binom{k}{s} n^{s-1}$, since one of the $s$ nodes must be $i$, and $|R_{ij}^*(k,s,1)| \leq \binom{k}{s}n^{s-2}$. Finally, note that if $c\in R_{ij}^*(k,s,1)$ then $|c|\geq s$, and even more, if $c\in R_{ij}^*(k,s,0)$, then $|c|\geq s+1$. Indeed, since $j\notin V_k(c)$ and $c$ must have a cycle with nodes in $V_k(c)$, there must be at least one node with two exiting edges: the node in the cycle that escapes from the cycle to connect with $j$. With this in mind, we can write
\begin{align*}
   \sum_{c\in R^*_{ij}(k)} \EE(H_c)
   &= \sum_{c\in R^*_{ij}(k)}  \EE(H_c|\ind_c=1)\PP(\ind_c=1) \\
   &\leq \sum_{c\in R^*_{ij}(k)}  \frac{p^{|c|}k!}{p^{k}(n-k)^{k}} \\
   &= \sum_{s=2}^{k-1}\left(\sum_{c\in R^*_{ij}(k,s,0)} \frac{p^{|c|}(k!)}{p^{k}(n-k)^{k}}  + \sum_{c\in R^*_{ij}(k,s,1)} \frac{p^{|c|}(k!)}{p^{k}(n-k)^{k}}\right)\\
   &\leq k!\sum_{s=2}^{k-1}\left(\sum_{c\in R^*_{ij}(k,s,0)} \frac{p^{s+1}}{p^{k}(n-k)^{k}}  + \sum_{c\in R^*_{ij}(k,s,1)} \frac{p^s}{p^{k}(n-k)^{k}}\right)\\
   &\leq k!\sum_{s=2}^{k-1}\left( \frac{\binom{k}{s}n^{s-1}}{(n-k)^{s+1}}\left(\frac{1}{p(n-k)}\right)^{k-(s+1)}  + \frac{\binom{k}{s}n^{s-2}}{(n-k)^{s}}\left(\frac{1}{p(n-k)}\right)^{k-s}\right)\\
   &\leq \frac{2(k!)}{n^2(1-k/n)^{k}} \sum_{s=2}^{k} \binom{k}{s}\left(\frac{1}{p(n-k)}\right)^{k-s}\\
    &= \frac{2(k!)}{n^2(1-k/n)^{k}}\left( 1+ \frac{1}{p(n-k)} \right)^k
\end{align*}
By taking $n$ sufficiently large, we have that $\sum_{c \in R_{ij}(k)}|\EE(H_{c}) - \EE(H)_{c}|$ is upper bounded by
\begin{align}
 &\frac{(k-1)(k+2)}{2(n-1)^{2}(1-1/n)^{k-2}}+ (1-p)^{n-1} + \frac{2(k!)}{n^2(1-k/n)^{k}}\left( 1+ \frac{1}{p(n-k)} \right)^k\leq \frac{C_1(k)}{n^2}.\label{eq-InProof:BoundForRep}
\end{align}
Now, fix $c\in P_{ij}(k)$. We will study upper and lower bounds of $\E(H_{c}) - \E(H)_{c}$. Note first that, since all nodes in $c$ are different we have that $\{\ind_{c_r,c_{r+1}}\, :\, r\in [k]\}$ are mutually independent. So, we can write 
\begin{align*}
     \E(H_{c}) - \E(H)_{c} &= \E(H_{c}|\ind_c=1)\PP(\ind_c=1) - \prod_{r=1}^k\E(H_{c_r,c_{r+1}}|\ind_{c_r,c_{r+1}}=1)\PP(\ind_{c_r,c_{r+1}}=1)\\
     &= \left(\E(H_{c}|\ind_c=1) - \prod_{r=1}^k\E(H_{c_r,c_{r+1}}|\ind_{c_r,c_{r+1}}=1)\right)\prod_{r=1}^k\PP(\ind_{c_r,c_{r+1}}=1)\\
     &= \left(\prod_{r=1}^k \alpha_{c_r}\right)\left(\E\left( \prod_{r=1}^k \td_{c_r} \big|\ind_c=1\right) - \prod_{r=1}^k\E(\td_{c_r}|\ind_{c_r,c_{r+1}}=1)\right) p^k.
\end{align*}
Then, applying the first upper bound from Lemma~\ref{lemma:Upper-Bound-fi-MultipleEdges} to $\E(H_c)$ and the lower bound from Lemma~\ref{lemma:Lower-Bound-fi-MultipleEdges} to $E(H)_c$, we have that for every $n$ large enough
    \begin{align*} 
       \E(H_{c}) - \E(H)_{c}  &\leq \left(\prod_{r=1}^k \alpha_{c_r}\right)\left(\frac{1}{(n-k)^kp^k} - \prod_{r=1}^k\frac{1}{np}\right) p^k
        \leq  \frac{1}{(n-k)^{k}} - \frac{1}{n^{k}}.
    \end{align*}
    Note that the leading term is the same for both polynomials $(n-k)^k$ and $n^k$. Thus, the polynomial $\rho(n) = n^k - (n-k)^k$ is of degree at most $k-1$, that is, $\rho\in \mathcal{P}_{k-1}[\R]$. Then, we can write
    \begin{align*}
         \frac{1}{(n-k)^{k}} - \frac{1}{n^{k}}  & = \frac{n^{k}-(n-k)^{k} }{n^{k}(n-k)^{k}} = \frac{\rho(n)}{n^{k}(n-k)^{k}} \leq \frac{\rho(n)}{(n-k)^{2k}}.
    \end{align*}
    Since $\rho(n)$ is a polynomial of degree $k-1$, there exists a constant $M_1(k)>0$ large enough such that $\rho(n)\leq M_1(k)n^{k-1}$. By adjusting $M_1(k)$ if necessary, we conclude that
    \begin{equation}\label{eq-InProof:UB-Diff-Hc} 
     \E(H_{c}) - \E(H)_{c}\leq \frac{\rho(n)}{(n-k)^{2k}} \leq \frac{M_1(k)}{n^{k+1}}.  
    \end{equation}
    Now, applying the lower bound of Lemma \ref{lemma:Lower-Bound-fi-PathsWithoutRep} to $\E(H_c)$ and the upper bound from Lemma \ref{lemma:Upper-Bound-fi-OneEdge} to $\E(H)_c$, we have that for every $n$ large enough
    \begin{align*} 
       \E(H_{c}) - \E(H)_{c}  &\geq  \left(\prod_{r=1}^k \alpha_{c_r}\right)\left(\frac{1}{(n-k)^kp^k}\left(1-\frac{1}{\log(n)}\right)^k - \prod_{r=1}^k\frac{1}{(n-1)p}\right) p^k\\
       &= \left(\prod_{r=1}^k \alpha_{c_r}\right)\left(\frac{1}{(n-k)^k}\left(1-\frac{1}{\log(n)}\right)^k - \frac{1}{(n-1)^k}\right).
    \end{align*}
    Let us denote $\delta= \left(1-{1}/{\log(n)}\right)^k$. Observe that, using similar arguments on the polynomials degrees, there exists $M_2(k)>0$ such that 
    \begin{align*}
        1-\delta = \frac{\log(n)^k - (\log(n)-1)^k}{\log(n)^k} \leq \frac{M_2(k)}{\log(n)}.
    \end{align*}
    Since the polynomials $(n-k)^k$ and $(n-1)^k$ have the same leading term, again with similar arguments, we deduce there exist $M_3(k),M_4(k)>0$ such that
    \begin{align*}
        \frac{1}{(n-k)^k}\delta - \frac{1}{(n-1)^k} &=  \delta\left( \frac{1}{(n-k)^k} - \frac{1}{(n-1)^k}\right) - \left(1 - \delta\right)\frac{1}{(n-1)^k}\\
        &\geq -\delta \left( \frac{M_3(k)}{n^{k+1}} \right) - \frac{M_2(k)}{(n-1)^k\log(n)}\\
        &\geq - \frac{M_3(k)}{n^{k}\log(n)}  - \frac{M_2(k)}{(n-1)^k\log(n)} \geq -\frac{M_4(k)}{n^k\log(n)}.
    \end{align*}
    Thus,  
    \begin{equation}\label{eq-InProof:LB-Diff-Hc}
    \E(H_{c}) - \E(H)_{c}  \geq  -\left(\prod_{r=1}^k \alpha_{c_r}\right)\frac{M_4(k)}{\log(n)n^k} \geq -\frac{M_4(k)}{\log(n)n^k}.
    \end{equation}
    By taking $C_2(k) = \max\{M_1(k),M_4(k)\}$ and mixing equations \eqref{eq-InProof:UB-Diff-Hc} and \eqref{eq-InProof:LB-Diff-Hc}, we conclude that 
    \[
    |\E(H_c) - \E(H)_c|\leq \frac{C_2(k)}{n^k\log(n)}.
    \]
    Finally, we see how many trails exist without repetition over its first $k$ nodes, i.e., the cardinality of $P_{ij}(k)$. A trail of size $k$ that induces a path of length $k$ (i.e. with $k$ different edges) involves $k+1$ nodes. However, since the nodes $i$ and $j$ are already fixed, we are left with $k-1$ nodes to choose. Thus, there are at most $n^{k-1}$ distinct trails without repetitions on their first $k$ nodes. Therefore, 
    \begin{equation}\label{eq-InProof:boundForNotRep}
        \sum_{c \in P_{ij}(k)}|\E(H_{c}) - \E(H)_{c}| = n^{k-1}\frac{C_2(k)}{n^{k}\log(n)}\leq \frac{C_2(k)}{n\log(n)}.
    \end{equation}
By mixing this inequality with \eqref{eq-InProof:BoundForRep}, we deduce that for sufficiently large $n$, 
\begin{align*}
     |\EE(H_{ij}^{k}) - \EE(H)_{ij}^{k}|&=\sum_{c \in R_{ij}(k)} |\E(H_c) - \EE(H)_{c}| + \sum_{c \in P_{ij}(k)} |\E(H_c) - \EE(H)_{c}|\\
     &\leq \frac{C_1(k)}{n^2} + \frac{C_2(k)}{n\log(n)} \leq (C_1(k) + C_2(k))\frac{1}{n\log(n)}. 
\end{align*}
The proof of \eqref{eq:Bound-by-coordinates} is completed considering $C(k) = C_1(k) + C_2(k)$. To finish, we have that
\begin{align*}
    \| \EE(H^{k})-\EE(H)^{k}\|_{\ast} &= \displaystyle \max_{i\in [n]} \displaystyle \sum_{j = 1}^{ n} | \EE(H_{ij}^{k}) - \EE(H)_{ij}^{k} |
    \leq \displaystyle \max_{i\in [n]} \displaystyle \sum_{j = 1}^{ n}  \frac{C(k)}{n\log(n)} = \frac{C(k)}{\log(n)},
\end{align*}
which goes to zero as $n\to \infty$. This finished the proof.
\end{proof}

We are now ready to complete the proof of our main results.
\begin{proof}[Proof of Theorem \ref{thm:Main-Matrix}]
Let $\phi(x)=\sum_{k=0}^{\infty}b_kx^k$ be an analytic function with convergence radius $\rho$. Recall from \eqref{desig.sum||E(H^k) - E(H)^k||*_cap3} we have
\[
     \|\EE(\phi(H)) - \phi(\EE(H)) \|_{\ast} \leq  \displaystyle\limsup_{t\to \infty} \sum_{k=0}^{t-1} b_k\|\EE(H^{k})-\EE(H)^{k}\|_{\ast}
\]

Fix a value $\varepsilon>0$. From Assumption \ref{Assumption-UniformBound} there is a uniform bound such that for all $n\in\N$, $\|H\|_{\ast}\leq \bar{\alpha}<1$. Choose $m_0 = m_0(\varepsilon, \bar{\alpha})$ such that $2\sum_{k=m_0+1}^{\infty}b_k\bar{\alpha}^k \leq \varepsilon$. Then, for every $n\in\N$,

\begin{align}
 \textstyle\sum_{k=m_0 + 1}^{\infty} b_k\left\|  \EE(H^{k}) - \EE(H)^{k} \right\|_{\ast} &\leq  \textstyle\sum_{k=m_0 + 1}^{\infty} b_k\left(\left\|  \EE(H^{k}) \right\|_{\ast} + \left\|  \EE(H)^{k} \right\|_{\ast}\right) \nonumber \\
 & \leq \textstyle\sum_{k=m_0 + 1}^{\infty}  b_k(\EE(\left\| H\right\|_{\ast}^{k}) + \EE(\left\|  H\right\|_{\ast}) ^{k}) \label{eq-InProof:InequalityFromSubmultiplicativity} \\
  & = \textstyle\sum_{k=m_0 + 1}^{\infty} 2 b_k\bar{\alpha}^{k} \leq \varepsilon, \nonumber
\end{align}
where the inequality \eqref{eq-InProof:InequalityFromSubmultiplicativity} follows from the submultiplicative property of $\|\cdot\|_{\ast}$ (see comments below \eqref{eq:Norm-Star}). Then, applying Proposition \ref{prop:DifferenceForFixedExponent}, we have
\begin{align*}
    \displaystyle \limsup_{ n \to \infty} \|\EE(\phi(H)) - \phi(\EE(H)) \|_{\ast}  & \leq \displaystyle \limsup_{ n \to \infty}\displaystyle \limsup_{ t \to \infty} \sum_{k=0}^{t-1} b_k\left\|  \EE(H^{k}) - \EE(H)^{k} \right\|_{\ast}  \\
      &\leq  \displaystyle \lim_{ n \to \infty}\sum_{k=0}^{m_0} b_k\left\|  \EE(H^{k}) - \EE(H)^{k} \right\|_{\ast} + \varepsilon = \varepsilon
\end{align*}
Since $\varepsilon>0$ is arbitrary, this finishes the proof.
\end{proof}

\section{Opinion Dynamics in Directed Graphs}\label{sec:RemarkDirected}

In this section, we consider the case where a simple directed graph captures the underlying network for the opinion dynamics. Directed graphs are used to capture asymmetric interactions between agents. 
Denoting $\mathbb{D}_n$ as the set of all directed graphs with a set of nodes $[n]$, let the measurable space $(\mathbb{D}_{n},\mathcal{P}(\mathbb{D}_{n}))$ and, similar to  \eqref{def:ERG-undirected}, endow it with the (unique) probability measure $\PP$ satisfying that
\begin{equation}
    \PP(G) = \displaystyle \prod_{e \in E(G)} p \displaystyle \prod_{e \notin E(G)} (1-p),
\end{equation}
for every directed graph $G\in \mathbb{D}_{n}$. The probability space $(\mathbb{D}_{n},\mathcal{P}(\mathbb{D}_{n}),\PP)$ obtained is the Erd\H{o}s-Rényi model of random directed graphs induced by $p$, that we denote by $\mathbb{D}(n,p)$. For each agent $i \in [n]$, the set of outgoing neighbours is $N^{+}(i) = \{ j\in [n]: (i,j)\in E(G)\}$ and the out-degree is equal to $\deg^+(i) = |N^{+}(i)|$. The agent $i$ would be considered isolated if its out-degree is zero.
In the directed counterpart of the opinion dynamics \eqref{eq:OpinionModel}, the opinion of agent $i$ at time $t+1$ is given by
\begin{equation}\label{eq:OpinionModel-directed}
    x_{i}(t+1) =
\begin{cases}
 \alpha_{i} \sum_{j \in N^{+}(i)}x_j(t)/\deg^{+}(i) + (1- \alpha_{i} )x_{i}(0) & \text{if } \deg^+(i)> 0, \\
 \alpha_{i}x_i(t) + (1 - \alpha_{i})x_i(0) & \text{if } \deg^+(i)=0, 
\end{cases}
\end{equation}
where $x(0) \in [0,1]^n$ is the vector of initial opinions.

For any graph $G\in \mathbb{D}_n$, the system \eqref{eq:OpinionModel-directed} can be written as $x(G,t+1) = H(G)x(G,t) + Bx(0)$, where $B=\mathrm{diag}(1-\alpha_i)$, and the matrix $H(G)$ is defined analogously as for \eqref{SLD}, that is,
\begin{equation*}
H_{ij}(G) =
    \begin{cases}
        \alpha_{i}A_{ij}(G)/\deg^+(i) & \text{if } \deg^+(i) > 0, \\
        \alpha_i\cdot 1_{\{i=j\}} & \text{if } \deg^+(i) = 0.  
    \end{cases}
\end{equation*}
Here, $A(G)$ is the (asymmetric) adjacency matrix of graph $G$. Finally, the mean-field dynamics is given by considering the expected matrix $\EE(H(G))$ instead of the random matrix $H(G)$.

\subsection{Mean-field Approximation for the $\ell_{\infty}$-norm}\label{sec:directed-infty}
 
It is direct to verify that Propositions \ref{Prop.NodoAislado} and \ref{prop:LimitOpinion} remain true for the directed model. Similarly, Lemma \ref{lemma:Upper-Bound-fi-OneEdge} and \ref{lemma:Upper-Bound-fi-MultipleEdges} hold true by considering
\begin{equation}\label{funciónfiGND-Directed}
    \td_i = \begin{cases}
        1/\deg^+(i) & \text{ if } \deg^+(i) >0, \\
        0 & \text{ if } \deg^+(i) = 0.
    \end{cases}
\end{equation}
Note that for $G\sim\mathbb{D}(n,p)$, the random variables $\{f_i:i\in[n]\}$ are mutually independent. Indeed, for every pair $i,j\in [n]$, the random variable $\ind_{ij}$ is know separated in $\ind_{(i,j)}$, which is aggregated on $\deg^+(i)$ and therefore in $f_i$, and in $\ind_{(j,i)}$, which is aggregated on $\deg^+(j)$ and therefore in $f_j$. The mutual independency is then deduced as an application of Proposition \ref{prop:independienciaMutuaIndicatrices}. With this in mind, it is possible to obtain a directed version for the mean-field approximation result of Theorem \ref{thm:Main-Opinions}, by means of the following directed version of Proposition \ref{prop:DifferenceForFixedExponent}.

\begin{proposition}\label{prop:DifferenceForFixedExponent-Directed}
Let $G \sim \mathbb{D}(n,p)$ and suppose that $p\gg \log(n)/n$. Then, for $k\in\N$ fixed and $n\in\N$ large enough,
\(
\|\EE(H^{k}) - \EE(H)^{k} \|_{\ast}\leq \frac{C(k)}{n} \to 0
\) as $n\to \infty$.
\end{proposition}
\begin{proof}
 As in proof of Proposition \ref{prop:DifferenceForFixedExponent}, let $P_{ij}(k)$ be the set of all trails in  $C_{ij}(k)$ without repetitions on its first $k$ nodes and let $R_{ij}(k)$ be the set of all trails in  $C_{ij}(k)$ with at least one repetition on its first $k$ nodes. Again, denote by $H_c$ the product of the random variables that constitute the directed trail $c$ of size $k$, that is,
\(
H_c = \prod_{r=1}^{k} H_{c_rc_{r+1}}.
\)
Let $\EE(H)_c$ denote the product of the $k$ expected values of the random variables that make up the directed trail $c$, that is, 
\(
\EE(H)_c = \prod_{r=1}^{k} \E(H_{c_rc_{r+1}}).
\)
We  again can write 
\begin{align*}
    |\EE(H_{ij}^{k}) - \EE(H)_{ij}^{k}| &= \textstyle\left|\sum_{c \in C_{ij}(k)}\EE(H_{c}) - \EE(H)_{c}\right|\nonumber\\
    &\leq \textstyle\sum_{c \in R_{ij}(k)}|\EE(H_{c}) - \EE(H)_{c}| + \sum_{c \in P_{ij}(k)} |\EE(H_{c}) - \EE(H)_{c}|.
\end{align*} 
Since Lemmas \ref{lemma:Upper-Bound-fi-OneEdge} and \ref{lemma:Upper-Bound-fi-MultipleEdges} hold for the directed model $\mathbb{D}(n,p)$ as well, one can replicate the development of the proof of Proposition \ref{prop:DifferenceForFixedExponent} to deduce \eqref{eq-InProof:BoundForRep}. That is, for every $n\in \N$ large enough,
\begin{equation}\label{eq-InProof:BoundR-ij-Directed}
\textstyle\sum_{c \in R_{ij}(k)}|\EE(H_{c}) - \EE(H)_{c}| \leq \frac{C(k)}{n^2}.
\end{equation}
for some constant $C(k)$. Now, for $c\in P_{ij}(k)$, we have that the random variables $\{ \td_{c_1},\ldots,\td_{c_k}\}$ are mutually independent. Indeed, since the nodes $c_1,\ldots,c_k$ are all different, the sets $K_r = \{ (c_r, s)\, :\, s\in [n]\setminus\{c_r\} \}$ for $r\in[k]$ are pairwise disjoint. Then, the mutual independence follows as an application of Proposition \ref{prop:independienciaMutuaIndicatrices}. Then, we have
\[
\E(H_c) = \E\left( \prod_{r=1}^{k} \alpha_{c_r} \td_{c_r}\ind_{c_rc_{r+1}}\right) =  \prod_{r=1}^{k} \alpha_{c_r} \E\left( \td_{c_r}\ind_{c_rc_{r+1}}\right) = \E(H)_c.
\]
We conclude then that $\sum_{c \in P_{ij}(k)} |\EE(H_{c}) - \EE(H)_{c}| = 0$ and so the conclusion follows.
\end{proof}

The latter proposition allows us to derive that Theorem \ref{thm:Main-Matrix} is also valid for the directed model (with the same proof), and therefore we can write the announced mean-field approximation result.
\begin{theorem}\label{thm:Main-Opinions-Directed}
    Let $G\sim \mathbb{D}(n,p)$. Suppose that Assumption \ref{Assumption-UniformBound} holds and that $p\gg \log(n)/n$. Then, 
    \begin{equation*}
       \lim_{n \to \infty} \| \E\left(x(G,\infty)\right) - \bar{x}(\infty) \|_{\infty} = 0.
    \end{equation*}
\end{theorem}
As we can appreciate, the study of the opinion models of Section \ref{sec:Models} over directed graphs is much simpler due to the edges symmetry-breaking and the consequent mutual independence of the out-degrees. Moreover, the rate of convergence of $\|\EE(H^k)-\EE(H)^k\|_{\ast}$ is improved from $O(1/\log(n))$ in Proposition \ref{prop:DifferenceForFixedExponent}, to $O(1/n)$ in Proposition \ref{prop:DifferenceForFixedExponent-Directed}. Finally, it is worth to mention that in the directed model, the threshold of connectivity $p\gg \log(n)/n$ is still necessary even though Lemmas \ref{lemma:Upper-Bound-fi-OneEdge} y \ref{lemma:Upper-Bound-fi-MultipleEdges} do not need it: it is used to control the amount of isolated nodes and derive the final bound \eqref{eq-InProof:BoundR-ij-Directed}.

\subsection{Mean-field Approximation for $\ell_{\rho}$-norm with $\rho\in (1,\infty)$}\label{sec:directed-infty}

To finish this section, we derive an extension of Theorem \ref{thm:Main-Opinions-Directed}, replacing the $\ell_{\infty}$-norm by any $\ell_{\rho}$-norm with $\rho>1$ by reinforcing the regime to $p\gg \log(n)^2/n$. For $\rho\in(1,\infty)$, the $\ell_{\rho}$-norm of a matrix $M =(m_{ij})\in\R^{n\times n}$ as
\begin{equation*}
   \textstyle \|M\|_{\rho} = \Big(\sum_{i,j\in[n]}|m_{ij}|^{\rho}\Big)^{1/\rho}.
\end{equation*}
Our last theorem is based on improving Lemma \ref{lemma:Upper-Bound-fi-MultipleEdges} to obtain an upper bound that is summable when $p\gg \log(n)^2/n$, at least up to a value $m_0(n) = O(\log(n))$. This is sufficient to replicate the proof strategy of Theorem \ref{thm:Main-Matrix} and compensate a factor $n^{1/\rho}$ appearing due to the renorming.

\begin{lemma}\label{lemma:Directed-Upper-Bound-fi-MultipleEdges}
Let $G \sim \mathbb{D}(n,p)$ and $k<n$. Let $c$ be a trail of size $k$ and let $V_k(c) = \{i_1,\ldots,i_s\}$. For each $r\in [s]$ let $m_r$ be the number of repetitions of $i_r$ within the first $k$ elements of $c$. Then,
\begin{align*}
\EE\left(\td_{i_1}^{m_{1}} \cdots \td_{i_{s}}^{m_{r}}\left| \ind_c\right. = 1 \right)  
&\leq \frac{k^{(k-|c|)}}{p^{k}(n-k)^{k}}.
\end{align*}
\end{lemma}

\begin{proof}
Since $\ind_c=1$, $c$ does not have two equal consecutive nodes. Let \[C_r = \{j \in V(c): j \text{ is a successor of }i_r\text{ in }c\}\] and let $\overline{C}_r = [n] \setminus (C_r \cup \{c_r\})$. Let us consider the random variable 
\begin{align*}
\deg(\overline{C}_r)= & |\{(i_r,j)\in E(G) : j\in \overline{C}_{r}\}| = \deg^+(i_r) - |C_r|.
\end{align*}

Note that in this case the variables $\{\ind_c,\deg(\overline{C}_1),\ldots,\deg(\overline{C}_s)\}$ are mutually independent, thanks to Proposition \ref{prop:independienciaMutuaIndicatrices}. Thus, denoting $l_r = |C_r|$ we have
\begin{align}
    \EE\left(\td_{i_1}^{m_{1}} \cdots \td_{{i}_s}^{m_{s}}\left|\ind_c\right. = 1 \right) &= \EE\left( \prod_{r=1}^{s}\left(\left.\frac{1}{|C_r| + \deg(\overline{C}_{r})}\right)^{m_{r}} \right| \ind_c = 1 \right)\nonumber   \\
&=\prod_{r=1}^{s} \E\left( \left(\frac{1}{l_r + \deg(\overline{C}_{r})}\right)^{m_r}\right). \nonumber
\end{align}

For a fixed $r\in [s]$, we analyze the term $\theta_r:=\E((1/(l_r + \deg(\overline{C}_{r})))^{m_{r}})$. Suppose first that $m_r\geq 2$. We will use two other inequalities holding for every $a\in\N$:
Firstly, $(k+a)^{a}k!\geq (k+a)!$, and secondly $(n - a)^{a} \leq \prod_{l=1}^{a}(n - a)$. Let $\Delta_r= m_r-l_r$ and observe that $\Delta_r\ge 0$ since the node $i_r$ appears at least once per element in $C_r$. Then, we write

\begin{align} 
    \theta_r&= \sum_{t = 0}^{n-1-l_r} \left(\frac{1}{l_r+t}\right)^{m_r} \binom{n-1-l_r}{t}p^{t}(1-p)^{n-1-l_r-t}\nonumber \\
    &\leq\sum_{t = 0}^{n-1-l_r} \left(\frac{m_r+t}{l_r+t}\right)^{\Delta_r} \left(\frac{1}{m_r+t}\right)^{m_r} \binom{n-1-l_r}{t}p^{t}(1-p)^{n-1-l_r-t}\nonumber \\
    &\leq \sum_{t = 0}^{n-1-l_r}\left(\frac{m_r+t}{l_r+t}\right)^{\Delta_r}\frac{(n-1-l_r)!}{(t+m_r)!(n-1-l_r-t)!}p^{t}(1-p)^{n-1-l_r-t}\nonumber \\
    &\leq \frac{1}{(n-l_r)^{m_r}}\sum_{t = 0}^{n-1-l_r}\left(\frac{m_r+t}{l_r+t}\right)^{\Delta_r}\frac{(n+\Delta_r-1)!}{(t+m_r)!(n+\Delta_r-1-(t+m_r))!}p^{t}(1-p)^{n+\Delta_r-1-(t+m_r)}\nonumber \\
    &= \frac{1}{(n-l_r)^{m_r}}\sum_{t = 0}^{n-1-l_r}\left(\frac{m_r+t}{l_r+t}\right)^{\Delta_r}\binom{n+\Delta_r-1}{t+m_r}p^{t}(1-p)^{n+\Delta_r-1-(t+m_r)}\nonumber \\
    &= \frac{1}{p^{m_r}(n-l_r)^{m_r}}\sum_{t = m_r}^{n+\Delta_r-1}\left(\frac{t}{t-\Delta_r}\right)^{\Delta_r}\binom{n+\Delta_r-1}{t}p^{t}(1-p)^{n+\Delta_r-1-t}\nonumber\\
   &\leq \left(\frac{m_r}{l_r}\right)^{\Delta_r}\frac{1}{p^{m_r}(n-l_r)^{m_r}} \leq \frac{m_r^{\Delta_r}}{p^{m_r}(n-k)^{m_r}}.\nonumber
\end{align}
Note that the same inequality also holds for $m_r=1$, by Lemma \ref{lemma:Upper-Bound-fi-OneEdge} applied to directed graphs. We deduce that
\begin{align}
    \prod_{r=1}^{s} \E\left( \left(\frac{1}{l_r + \deg(\overline{C}_{r})}\right)^{m_r}\right) &\leq \prod_{r=1}^{s}\frac{m_r^{\Delta_r}}{(n-k)^{m_r}p^{m_r}} \nonumber \\
    &\leq  \frac{k^{\Delta_1+\cdots+\Delta_s}}{p^{(m_{1}+\cdots +m_{s})}(n-k)^{(m_{1}+\cdots +m_{s})}} \nonumber \\
    & = \frac{k^{(k-|c|)}}{p^{k}(n-k)^{k}}, \nonumber
\end{align}
where the last inequality considers that $l_{1} + \cdots + l_{s} = |c|$. The proof is then completed.
\end{proof}
With this enhanced upper bound for trails with repetitions, we derive a new bound for $|\EE(H_{ij}^{k}) - \EE(H)^{k}_{ij}|$ by using the regime $p\gg \log(n)^2/n$.
\begin{proposition}\label{prop:DifferenceForFixedExponent-DirectedEnhanced}
Let $G \sim \mathbb{D}(n,p)$, suppose that $p\gg \log(n)^2/n$, and let $R>0$. Then, there exists $n_0\in\N$ such that for every $n\geq n_0$, for every $k \leq R\log(n)$, and every $i,j\in [n]$, we have 
\begin{equation}\label{eq:Bound-by-coordinates-DirectedEnhanced}
    |\EE(H_{ij}^{k}) - \EE(H)^{k}_{ij}| \leq (1-p)^{n-1}+ \frac{2eR\log(n) + R^2\log(n)^2}{(n-R\log(n))^2 (1-R\log(n)/n)^{R\log(n)}}.
\end{equation}
\end{proposition}

\begin{proof}
    As in proof of Proposition \ref{prop:DifferenceForFixedExponent}, let $P_{ij}(k)$ be the set of all trails in  $C_{ij}(k)$ without repetitions on its first $k$ nodes, and let $R_{ij}(k)$ be the set of all trails in  $C_{ij}(k)$ with at least one repetition on its first $k$ nodes. Again, denote by $H_c$ the product of the random variables that constitute the directed trail $c$ of size $k$, that is,
\(
H_c = \prod_{r=1}^{k} H_{c_rc_{r+1}}.
\)
Let $\EE(H)_c$ denote the product of the $k$ expected values of the random variables that make up the directed trail $c$, that is, 
\(
\EE(H)_c = \prod_{r=1}^{k} \E(H_{c_rc_{r+1}}).
\)
As in the proof of Proposition \ref{prop:DifferenceForFixedExponent},
\begin{align*}
|\EE(H_{ij}^{k}) - \EE(H)_{ij}^{k}| &\leq \sum_{c \in R_{ij}(k)}|\EE(H_{c}) - \EE(H)_{c}|\leq \frac{|R_{ij}(k)|}{(n-1)^k}+(1-p)^{n-1} + \sum_{c\in R^*_{ij}(k)}\EE(H_c),
\end{align*}
where $R^*_{ij}(k)$ are all trails in $R_{ij}(k)$ without two equal consecutive nodes.  Again, we can deduce that
\begin{align}
    \frac{|R_{ij}(k)|}{(n-1)^k}\leq \frac{(k-1)(k+2)}{2(n-1)^{2}(1-1/n)^{k-2}} \leq \frac{(k-1)(k+2)}{2(n-k)^{2}(1-k/n)^{k}},\nonumber
\end{align}
where the second inequality is just a less tight bound. Again, as in proof of Proposition \ref{prop:DifferenceForFixedExponent}, let us denote
\begin{align*}
R_{ij}^*(k,s,0) &=\{ c\in R_{ij}^*(k)\ :\ |V_k(c)| = s, j\notin V_k(c) \},\\
R_{ij}^*(k,s,1) &=\{ c\in R_{ij}^*(k)\ :\ |V_k(c)| = s, j\in V_k(c) \}.
\end{align*}
Recall from the construction that $R_{ij}^*(k) = \bigcup_{s=2}^{k-1}(R_{ij}^*(k,s,0)\cup R_{ij}^*(k,s,1))$, that $|R_{ij}^*(k,s,0)| \leq  \binom{k}{s} n^{s-1}$, and that $|R_{ij}^*(k,s,1)| \leq \binom{k}{s}n^{s-2}$. Finally, recall that for $c\in R_{ij}^*(k,s,1)$ we have $|c|\geq s$ and if $c\in R_{ij}^*(k,s,0)$, then $|c|\geq s+1$. With this in mind, we can write
\begin{align*}
   \sum_{c\in R^*_{ij}(k)}\EE(H_c) &= \sum_{s=2}^{k-1}\left(\sum_{c\in R^*_{ij}(k,s,0)} \EE(H_c|\ind_c=1)p^{|c|} + \sum_{c\in R^*_{ij}(k,s,1)} \EE(H_c|\ind_c=1)p^{|c|}\right)\\
   &\leq \sum_{s=2}^{k-1}\left(\sum_{c\in R^*_{ij}(k,s,0)} \frac{p^{|c|}k^{k-|c|}}{p^{k}(n-k)^{k}}  + \sum_{c\in R^*_{ij}(k,s,1)} \frac{p^{|c|}k^{k-|c|}}{p^{k}(n-k)^{k}}\right)\\
   &\leq \sum_{s=2}^{k-1}\left(\sum_{c\in R^*_{ij}(k,s,0)} \frac{p^{s+1}k^{k-(s+1)}}{p^{k}(n-k)^{k}}  + \sum_{c\in R^*_{ij}(k,s,1)} \frac{p^sk^{k-s}}{p^{k}(n-k)^{k}}\right)\\
   &\leq \sum_{s=2}^{k-1}\left(\binom{k}{s}n^{s-1}\frac{p^{s+1}k^{k-(s+1)}}{p^{k}(n-k)^{k}}  +  \binom{k}{s}n^{s-2}\frac{p^sk^{k-s}}{p^{k}(n-k)^{k}}\right)\\
   &\leq \frac{1}{(n-k)^2} \sum_{s=2}^{k-1}\binom{k}{s}\left( \frac{1}{(1-\tfrac{k}{n})^{s-1}}\left(\frac{k}{p(n-k)}\right)^{k-(s+1)} + \frac{1}{(1-\tfrac{k}{n})^{s-2}}\left(\frac{k}{p(n-k)}\right)^{k-s}\right)\\
   &\leq \frac{2}{(n-k)^2}\frac{1}{(1-k/n)^{k}} \sum_{s=2}^{k-1}\binom{k}{s}\left(\frac{k}{p(n-k)}\right)^{k-(s+1)}\\
   &\leq \frac{2}{(n-k)^2}\frac{k+1}{(1-k/n)^{k}} \sum_{s=2}^{k-1}\binom{k}{s+1}\left(\frac{k}{p(n-k)}\right)^{k-(s+1)}\\
   &\leq \frac{2}{(n-k)^2}\frac{k+1}{(1-k/n)^{k}}\left(1 + \frac{k}{p(n-k)}\right)^k.
\end{align*}

Now, choose $n_0\in \N$ large enough such that for all $n\geq n_0$ it holds $p(n-R\log(n))\geq R^2\log(n)^2$. Then, we can write
\begin{align*}
    \frac{|R_{ij}^*(k)|}{(n-k)^k}& \leq \frac{R^2\log(n)^2}{(n-R\log(n))^2 (1-R\log(n)/n)^{R\log(n)}},\\
    \sum_{R^*_{ij}(k)}\EE(H_c) &\leq \frac{2}{(n-R\log(n))^2}\cdot \frac{R\log(n)}{(1-R\log(n)/n)^{R\log(n)}} \underbrace{\left(1+\frac{1}{R\log(n)}\right)^{R\log(n)}}_{\leq \ e}
\end{align*}
Thus, conclude that 
\[
|\EE(H_{ij}^{k}) - \EE(H)_{ij}^{k}| \leq (1-p)^{n-1}+ \frac{2eR\log(n) + R^2\log(n)^2}{(n-R\log(n))^2 (1-R\log(n)/n)^{R\log(n)}}.
\]
\end{proof}
Now, we derive the announced theorem.
\begin{theorem}\label{thm:Main-Opinion-Extension} 
Let $G\sim\mathbb{D}(n,p)$. Suppose that Assumption \ref{Assumption-UniformBound} holds and $p\gg \log(n)^2/n$. Then, for any $\rho\in (1,\infty)$ it holds that
    \[
    \lim_{n\to\infty} \|\E(x(G,\infty)) - \bar{x}(\infty)\|_{\rho} = 0.
    \]
\end{theorem}
\begin{proof}
Note first that for any matrix $M\in\R^{n\times n}$ and every $x\in \R^n$, one has that
\begin{align*}
    \|Mx\|_{\rho} \leq n^{1/\rho}\|Mx\|_{\infty} \leq n^{1/\rho}\|M\|_{\ast}\|x\|_{\infty}.
\end{align*}
Thus, for every $\ell\in\N$, we have $\| \E(H^{\ell})Bx(0) - \E(H)^{\ell}Bx(0) \|_{\rho} \leq 2n^{1/\rho}\bar{\alpha}^{\ell}$ and therefore
for $m_0 \in \N$ it holds
\begin{align*}
  \sum_{k=m_0}^{\infty}\| \E(H^k)Bx(0) - \E(H)^kBx(0) \|_{\rho} &\sum_{k=0}^{\infty}\| \E(H^{k+m_0})Bx(0) - \E(H)^{k+m_0}Bx(0) \|_{\rho} \\
  &\le \sum_{k=0}^{\infty}2n^{1/\rho}\bar{\alpha}^{k+m_0}\\
  &\leq   2n^{1/\rho}\bar{\alpha}^{m_0}\sum_{k=0}^{\infty}\bar{\alpha}^k\\
  &= \frac{2}{1-\bar{\alpha}}\exp\left(\frac{1}{\rho}\log(n) - m_0\log(1/\bar{\alpha})\right).
\end{align*}
Fix $\varepsilon>0$ and take $m_0(n) = \frac{1}{\log(1/\bar{\alpha})}\left(\frac{1}{\rho}\log(n) - \log((1-\bar{\alpha})\varepsilon/2)\right)$. Then, for every $n\in\N$,
\begin{align*}
    \|\E(x(G,\infty)) - \bar{x}(\infty)\|_{\rho} \leq \sum_{k=1}^{m_0(n)} \| \E(H^k)Bx(0) - \E(H)^kBx(0) \|_{\rho} + \varepsilon.
\end{align*}
Assume that $n_0\in\N$ is large enough such that for every $n\geq n_0$ one has that $m_0(n)< n$. Note that there exists a constant $R>0$ large enough such that
$m_0(n)+2 \leq R\log(n)$ for all $n\geq n_0$. Enlarging $n_0$ if necessary, we can assume that the upper bound of Proposition \ref{prop:DifferenceForFixedExponent-DirectedEnhanced} holds.  Let $\theta_k := \| \E(H^k)Bx(0) - \E(H)^kBx(0) \|_{\rho}$ and set $r = (\rho+1)/\rho$. Then, for every $k\leq m_0(n)$ we can write
\begin{align*}
 \theta_k &= \textstyle{\left(\sum_{i=1}^n \left(\sum_{j=1}^n (\E(H^k_{ij}) - \E(H)^k_{ij})(Bx(0))_j\right)^\rho\;\right)}^{1/\rho}\\
 &\leq {\left(n^{1+\rho} \left((1-p)^{n-1}+ \frac{2eR\log(n) + R^2\log(n)^2}{(n-R\log(n))^2 (1-R\log(n)/n)^{R\log(n)}}\right)^\rho\;\right)}^{1/\rho}\\
 &= n^{r}(1-p)^{n-1}+ \frac{n^{r}}{(n-R\log(n))^2}\frac{2eR\log(n) + R^2\log(n)^2}{ (1-R\log(n)/n)^{R\log(n)}}.
\end{align*}
Recalling that $\rho>1$ (and so $r<2$), we can write
\begin{align*}
    (1-R\log(n)/n)^{R\log(n)}\to 1,\quad
    \frac{(R^2\log(n)^2+2eR\log(n))n^r}{(n-R\log(n))^{2}} \to 0,\quad
    n^r(1-p)^{n-1}\to 0,
\end{align*}
as $n\to \infty$, where the last limit follows from $p\gg \log(n)^2/n$ (in fact, the limit still holds with $p\gg \log(n)/n$). Thus, we deduce that
\[
\sum_{k=1}^{m_0(n)} \| \E(H^k)Bx(0) - \E(H)^kBx(0) \|_{\rho}\to 0
\]
as $n\to \infty$. Finally, 
\begin{align*}
 \limsup_n\|\E(x(G,\infty)) - \bar{x}(\infty)\|_{\rho} \leq \lim_n\sum_{k=1}^{m_0(n)} \| \E(H^k)Bx(0) - \E(H)^kBx(0) \|_{\rho} + \varepsilon = \varepsilon,
\end{align*}
and since $\varepsilon>0$ is arbitrary, the result follows.
\end{proof}
\section{Final Remarks}\label{sec:remarks}

In this work, we consider opinion dynamics over Erd\H{o}s-Rényi random graphs. We provide a qualitative concentration result:  in both random graph models, directed and undirected, the expected value of the stable opinion converges to the stable opinion of the mean-field model as the size of the network grows. This result provides a tractable setting to compute the average stable opinion in large graphs, and in particular, it allows the tackling of mean value problems in stochastic optimization, where the random variable is given by the stable opinion on a large random graph. Our result complements \cite{XingJohansson2024Concentration}, where similar findings are obtained but for different opinion dynamics.

The scope of this work covers homogeneous Erd\H{o}s-Rényi random graphs to model the uncertain network and the $\ell_{\infty}$-norm to measure the gap between the stable opinions. For the model with directed graphs, the same result is derived $\ell_{\rho}$-norms with $\rho>1$ by paying a mild extra $\log(n)$ factor in the regime of the random graph. We believe that our results can be extended in different directions. For instance, one could consider non-homogeneous Erd\H{o}s-Rényi random graphs (the appearance of each edge is still independent, but with different probabilities) or more general graph models where the independence can be somehow localized. In particular, we think that geometric random graphs are an appealing direction to explore \cite{GeoRandmGraph2002,LibGeoRandmGraph2003}. However, since our methods heavily rely on the combinatorial analysis of homogeneous Erd\H{o}s-Rényi random graphs, such extension would require a different approach.

On the other hand, it is unclear if Theorem~\ref{thm:Main-Opinion-Extension} can be derived for undirected graphs. Similarly, regarding concentration inequalities, their development would require a more delicate study of the bounds we obtained. The central obstruction seems to be the lack of explicit bounds on the negative moments of the binomial distribution. An inductive approach using the recursive formula of \cite{ChaoStrawderman1972Negative} could be an alternative. There seems to be some room for improvement in this sense. 

\paragraph{Acknowledgements.} The first and second authors were partially funded by the project FONDECYT 11220586 (ANID-Chile). The second author was partially funded by Centro de Modelamiento Matem\'atico (CMM), FB210005, BASAL funds
for centers of excellence (ANID-Chile). The third author was partially funded by FONDECYT 1241846 (ANID-Chile) and ANILLO ACT210005 Information and Computation in Market Design (ANID-Chile).

\bibliographystyle{acm}
{\small \bibliography{manuscript/library}}

\begin{thebibliography}{10}

\bibitem{Abebe2018VaryingSusceptibility}
{\sc Abebe, R., Kleinberg, J., Parkes, D., and Tsourakakis, C.~E.}
\newblock Opinion dynamics with varying susceptibility to persuasion.
\newblock In {\em KDD\/} (2018), p.~1089–1098.

\bibitem{Acemouglu2013Opinion}
{\sc Acemoglu, D., Como, G., Fagnani, F., and Ozdaglar, A.}
\newblock Opinion fluctuations and disagreement in social networks.
\newblock {\em Mathematics of Operations Research 38}, 1 (2013), 1--27.

\bibitem{Acemoglu2011Bayesian}
{\sc Acemoglu, D., Dahleh, M.~A., Lobel, I., and Ozdaglar, A.}
\newblock Bayesian learning in social networks.
\newblock {\em Review of Economic Studies 78}, 4 (2011), 1201--1236.

\bibitem{Acemoglu2011Learning}
{\sc Acemoglu, D., and Ozdaglar, A.}
\newblock Opinion dynamics and learning in social networks.
\newblock {\em Dynamic Games and Applications 1}, 1 (2011), 3--49.

\bibitem{Amelkin2017Polar}
{\sc Amelkin, V., Bullo, F., and Singh, A.~K.}
\newblock Polar opinion dynamics in social networks.
\newblock {\em Institute of Electrical and Electronics Engineers. Transactions
  on Automatic Control 62}, 11 (2017).

\bibitem{Bala2000Noncooperative}
{\sc Bala, V., and Goyal, S.}
\newblock A noncooperative model of network formation.
\newblock {\em Econometrica 68}, 5 (2000), 1181--1229.

\bibitem{Ballester2005Who}
{\sc Ballester, C., Calv\'o-Armengol, A., and Zenou, Y.}
\newblock Who's who in networks. {W}anted: the key player.
\newblock {\em Econometrica 74}, 5 (2006), 1403--1417.

\bibitem{bauso2016opinion}
{\sc Bauso, D., Tembine, H., and Basar, T.}
\newblock Opinion dynamics in social networks through mean-field games.
\newblock {\em SIAM Journal on Control and Optimization 54}, 6 (2016),
  3225--3257.

\bibitem{Bernardo2024survey}
{\sc Bernardo, C., Altafini, C., Proskurnikov, A., and Vasca, F.}
\newblock Bounded confidence opinion dynamics: a survey.
\newblock {\em Automatica J. IFAC 159\/} (2024).

\bibitem{Bindel2011HowBad}
{\sc Bindel, D., Kleinberg, J., and Oren, S.}
\newblock How bad is forming your own opinion?
\newblock In {\em FOCS}. 2011, pp.~57--66.

\bibitem{Birge2011Introduction}
{\sc Birge, J.~R., and Louveaux, F.~c.}
\newblock {\em Introduction to stochastic programming}, second~ed.
\newblock Springer Series in Operations Research and Financial Engineering.
  Springer, New York, 2011.

\bibitem{byrd2016stochastic}
{\sc Byrd, R.~H., Hansen, S.~L., Nocedal, J., and Singer, Y.}
\newblock A stochastic quasi-newton method for large-scale optimization.
\newblock {\em SIAM Journal on Optimization 26}, 2 (2016), 1008--1031.

\bibitem{Chakraborti2023Majority}
{\sc Chakraborti, D., Kim, J.~H., Lee, J., and Tran, T.}
\newblock Majority dynamics on sparse random graphs.
\newblock {\em Random Structures \& Algorithms 63}, 1 (2023), 171--191.

\bibitem{ChaoStrawderman1972Negative}
{\sc Chao, M.~T., and Strawderman, W.~E.}
\newblock Negative moments of positive random variables.
\newblock {\em Journal of the American Statistical Association 67}, 338 (1972),
  429--431.

\bibitem{Chu2023Non-Markovian}
{\sc Chu, W., and Porter, M.~A.}
\newblock Non-{M}arkovian models of opinion dynamics on temporal networks.
\newblock {\em SIAM Journal on Applied Dynamical Systems 22}, 3 (2023),
  2624--2647.

\bibitem{GeoRandmGraph2002}
{\sc Dall, J., and Christensen, M.}
\newblock Random geometric graphs.
\newblock {\em Physical Review E 66}, 1 (2002), 016121.

\bibitem{deVos2024Influencing}
{\sc de~Vos, W., Borm, P., and Hamers, H.}
\newblock Influencing opinion networks: optimization and games.
\newblock {\em Dynamic Games and Applications 14}, 4 (2024), 959--980.

\bibitem{degroot1974reaching}
{\sc DeGroot, M.~H.}
\newblock Reaching a consensus.
\newblock {\em Journal of the American Statistical Association 69}, 345 (1974),
  118--121.

\bibitem{DiaconisFreedman1999Iterated}
{\sc Diaconis, P., and Freedman, D.}
\newblock Iterated random functions.
\newblock {\em SIAM Review 41}, 1 (1999), 45--76.

\bibitem{Diestel2018Graph}
{\sc Diestel, R.}
\newblock {\em Graph theory}, fifth~ed., vol.~173 of {\em Graduate Texts in
  Mathematics}.
\newblock Springer, Berlin, 2018.

\bibitem{Erdos-Renyi}
{\sc Erdös, P., and R{\'e}nyi, A.}
\newblock On the evolution of random graphs.
\newblock {\em Publications of the Mathematical Institute of the Hungarian
  Academy of Sciences 5}, 1 (1960), 17--60.

\bibitem{Ferraioli2016Descentralized}
{\sc Ferraioli, D., Goldberg, P.~W., and Ventre, C.}
\newblock Decentralized dynamics for finite opinion games.
\newblock {\em Theoretical Computer Science 648\/} (2016), 96--115.

\bibitem{folland1999real}
{\sc Folland, G.~B.}
\newblock {\em Real analysis: modern techniques and their applications}.
\newblock 1999.

\bibitem{Fotakis2023Limited}
{\sc Fotakis, D., Kandiros, V., Kontonis, V., and Skoulakis, S.}
\newblock Opinion dynamics with limited information.
\newblock {\em Algorithmica 85}, 12 (2023), 3855--3888.

\bibitem{FriedkinJohnsen1999Influence}
{\sc Friedkin, N.~E., and Johnsen, E.~C.}
\newblock Influence networks and opinion change.
\newblock {\em Advances in Group Processes 16}, 1 (1999), 1--29.

\bibitem{GhaderiSrikant2014Opinion}
{\sc Ghaderi, J., and Srikant, R.}
\newblock Opinion dynamics in social networks with stubborn agents: equilibrium
  and convergence rate.
\newblock {\em Automatica 50}, 12 (2014), 3209--3215.

\bibitem{Gilbert}
{\sc Gilbert, E.~N.}
\newblock Random graphs.
\newblock {\em The Annals of Mathematical Statistics 30}, 4 (1959), 1141--1144.

\bibitem{grimmett2020probability}
{\sc Grimmett, G., and Stirzaker, D.}
\newblock {\em Probability and random processes}, fourth~ed.
\newblock Oxford University Press, 2020.

\bibitem{hegselmann2002opinion}
{\sc Hegselmann, R., and Krause, U.}
\newblock Opinion dynamics and bounded confidence models, analysis and
  simulation.
\newblock {\em Journal of Artificial Societies and Social Simulation 5}, 3
  (2002).

\bibitem{Homem-de-Mello2014MonteCarlo}
{\sc Homem-de Mello, T., and Bayraksan, G.}
\newblock Monte {C}arlo sampling-based methods for stochastic optimization.
\newblock {\em Surveys in Operations Research and Management Science 19}, 1
  (2014), 56--85.

\bibitem{HornJohnson2013Matrix}
{\sc Horn, R.~A., and Johnson, C.~R.}
\newblock {\em Matrix analysis}, second~ed.
\newblock Cambridge University Press, Cambridge, 2013.

\bibitem{LibroRG1}
{\sc Janson, S., \L~uczak, T., and Rucinski, A.}
\newblock {\em Random graphs}.
\newblock Wiley-Interscience Series in Discrete Mathematics and Optimization.
  Wiley-Interscience, New York, 2000.

\bibitem{Jia2015Power}
{\sc Jia, P., MirTabatabaei, A., Friedkin, N.~E., and Bullo, F.}
\newblock Opinion dynamics and the evolution of social power in influence
  networks.
\newblock {\em SIAM Review 57}, 3 (2015), 367--397.

\bibitem{jiang2023opinion}
{\sc Jiang, H., Mazalov, V.~V., Gao, H., and Wang, C.}
\newblock Opinion dynamics control in a social network with a communication
  structure.
\newblock {\em Dynamic Games and Applications 13}, 1 (2023), 412--434.

\bibitem{Geethu2021Controllability}
{\sc Joseph, G., Nettasinghe, B., Krishnamurthy, V., and Varshney, P.~K.}
\newblock {Controllability of network opinion in Erd\H{o}s-R\'{e}nyi graphs
  using sparse control inputs}.
\newblock {\em SIAM Journal on Control and Optimization 59}, 3 (2021),
  2321--2345.

\bibitem{Kempe2015Spread}
{\sc Kempe, D., Kleinberg, J., and Tardos, E.}
\newblock Maximizing the spread of influence through a social network.
\newblock {\em Theory of Computing. An Open Access Journal 11\/} (2015),
  105--147.

\bibitem{koshal2016distributed}
{\sc Koshal, J., Nedi{\'c}, A., and Shanbhag, U.~V.}
\newblock Distributed algorithms for aggregative games on graphs.
\newblock {\em Operations Research 64}, 3 (2016), 680--704.

\bibitem{li2014efficient}
{\sc Li, M., Zhang, T., Chen, Y., and Smola, A.~J.}
\newblock Efficient mini-batch training for stochastic optimization.
\newblock In {\em ACM SIGKDD\/} (2014), pp.~661--670.

\bibitem{Lopez-Pintado2006Contagion}
{\sc L\'opez-Pintado, D.}
\newblock Contagion and coordination in random networks.
\newblock {\em International Journal of Game Theory 34}, 3 (2006), 371--381.

\bibitem{Meier2017Push}
{\sc Meier, F., and Peter, U.}
\newblock Push is fast on sparse random graphs.
\newblock {\em SIAM Journal on Discrete Mathematics 31}, 1 (2017), 29--49.

\bibitem{Mirtabatabaei2012Opinion}
{\sc Mirtabatabaei, A., and Bullo, F.}
\newblock Opinion dynamics in heterogeneous networks: convergence conjectures
  and theorems.
\newblock {\em SIAM Journal on Control and Optimization 50}, 5 (2012),
  2763--2785.

\bibitem{Munoz2025Exploring}
{\sc Mu\~noz, F.~J., Meacci, L., Nu\~no, J.~C., and Primicerio, M.}
\newblock Exploring the limits of the law of mass action in the mean field
  description of epidemics on {E}rd\"os-{R}\'enyi networks.
\newblock {\em Applied Mathematics and Computation 485\/} (2025).

\bibitem{musco2018minimizing}
{\sc Musco, C., Musco, C., and Tsourakakis, C.~E.}
\newblock Minimizing polarization and disagreement in social networks.
\newblock In {\em WWW\/} (2018), pp.~369--378.

\bibitem{Parsegov2017Multidimensional}
{\sc Parsegov, S.~E., Proskurnikov, A.~V., Tempo, R., and Friedkin, N.~E.}
\newblock Novel multidimensional models of opinion dynamics in social networks.
\newblock {\em Institute of Electrical and Electronics Engineers. Transactions
  on Automatic Control 62}, 5 (2017), 2270--2285.

\bibitem{LibGeoRandmGraph2003}
{\sc Penrose, M.}
\newblock {\em Random Geometric Graphs}.
\newblock Oxford University Press, 2003.

\bibitem{Shapiro2021Lectures}
{\sc Shapiro, A., Dentcheva, D., and Ruszczy\'nski, A.}
\newblock {\em Lectures on stochastic programming---modeling and theory},
  third~ed., vol.~28 of {\em MOS-SIAM Series on Optimization}.
\newblock SIAM, 2021.

\bibitem{XingJohansson2024Concentration}
{\sc Xing, Y., and Johansson, K.~H.}
\newblock Concentration in gossip opinion dynamics over random graphs.
\newblock {\em SIAM Journal on Control and Optimization 62}, 3 (2024),
  1521--1545.

\bibitem{Yang2024MinMax}
{\sc Yang, J., Zhou, L., Liu, J., Xi, J., and Zheng, Y.}
\newblock Min--max group consensus of discrete-time multi-agent systems under
  directed random networks.
\newblock {\em Systems \& Control Letters 193\/} (2024).

\end{thebibliography}
  
\end{document}